\documentclass[english,letterpaper,11pt,reqno]{amsart}
\usepackage{appendix}
\usepackage{amsbsy}
\usepackage{amsfonts}
\usepackage{amsmath}
\usepackage{amssymb}
\usepackage{graphicx}
\usepackage{amsthm}
\usepackage{graphicx}
\usepackage{ifthen}
\usepackage{relsize}
\usepackage{textcomp}
\usepackage{enumitem, color, amssymb}
\usepackage{mathrsfs}
\usepackage[bookmarksnumbered,colorlinks]{hyperref}
\hypersetup{colorlinks=true, linkcolor=blue, citecolor=blue}
\usepackage{placeins}
\usepackage{stackrel}

\newtheorem {lemma} {Lemma} [section]
\newtheorem{thm}{Theorem}
\newtheorem{theorem}{Theorem}

\newtheorem{prop}[lemma]{Proposition}

\theoremstyle{remark}
\newtheorem{remark}[lemma]{Remark}

\newcounter{nmdthmcnt}

\newcommand{\beqa}{\begin{eqnarray}}
\newcommand{\beq}{\begin{equation}}
\newcommand{\eeqa}{\end{eqnarray}}
\newcommand{\eeq}{\end{equation}}

\newcommand{\be}{\begin{equation}}
\newcommand{\ee}{\end{equation}}
\newcommand{\lb}[1]{\label{#1}}

\renewcommand{\Ref}[1]{(\ref{#1})}

\newcommand\kk{{\bf k}}

\newcommand\xx{{\bf x}}
\newcommand\yy{{\bf y}}
\newcommand\kf{\hat\kk}
\newcommand\tf{\hat\tT}
\newcommand\xf{\hat\xx}
\newcommand\yf{\hat\yy}

\newcommand\n{\nabla}

\newcommand{\HH}{\mathcal{H}}

\newcommand{\ta}{\tau}
\newcommand\om{\omega}
\newcommand\tT{{\bf t}}

\newcommand{\al}{\alpha}

\newcommand{\we}{\wedge}

\newcommand{\sig}{\sigma}

\newcommand{\ep}{\epsilon}
\newcommand{\lam}{\lambda}
\newcommand{\fr}{\frac}

\newcommand{\1}{\scriptscriptstyle 1}

\begin{document}
\title[]{Cohomogeneity one K\"ahler-Ricci solitons under a Heisenberg group action
and related metrics}
\author[]{Gideon Maschler and Robert Ream}
\address{Department of Mathematics and Computer Science\\ Clark University\\ Worcester, MA }
\email{Gmaschler@clarku.edu\,,\,Rream@clarku.edu}

\maketitle
\thispagestyle{empty}
\begin{abstract}
We show that integrability of an almost complex structure in complex dimension $m$
is equivalent, in the presence of an almost hermitian metric, to $m(m-1)$ equations
involving what we call shear operators. Inspired by this, we give an ansatz
for K\"ahler metrics in dimension $m>1$, for which at most $m-1$ of these shear equations
are non-trivial. The equations for gradient K\"ahler-Ricci solitons in this ansatz
are frame dependent PDEs, which specialize to ODEs under extra assumptions. Metrics
solving the latter system include a restricted class of cohomogeneity one metrics, and
we find among them complete expanding gradient K\"ahler-Ricci solitons under the action of
the ($2m-1$)-dimensional Heisenberg group, and some incomplete steady solitons.
We examine curvature properties and asymptotics for the former Ricci solitons.
In another special case of the ansatz we present, for $m=2$, a class of complete metrics
of a more general type which we call gradient K\"ahler-Ricci skew-solitons, which are
cohomogeneity one under the Euclidean plane group action. This paper continues research
started in \cite{mr,a-m2}.
\end{abstract}

\section{Introduction}

Ricci solitons have been studied extensively in the last few decades. They were
introduced by Hamilton \cite{h} as solutions  to the Ricci flow that vary only by rescaling
of pull-backs of an initial metric under diffeomorphisms of the manifold. In the case of
gradient solitons, that initial metric $g$ satisfies an equation of the form
\[ \mathrm{Ric}+\n df=\lam g,  \]
where $\mathrm{Ric}$ is the Ricci curvature of $g$ and $\n df$ is the Hessian of a smooth function $f$ on the given manifold. The soliton constant $\lam$ is negative in the expanding case, and the corresponding solution to the flow exists for finite negative time and all positive time.
This constant is zero in the steady case.

Explicit examples of K\"ahler-Ricci solitons were given in \cite{ca}, and those are both expanding and $U(n)$-invariant. Other early and more recent examples appear in \cite{ko,ca0,chv,g,ptv,fik,yan,fw,cd}. Among metrics on homogeneous manifolds, there is a well-known correspondence between left-invariant Einstein metrics on solvable Lie groups and homogeneous Ricci solitons on their nilradicals \cite{la0}.

The main result of this paper is the construction of {\em explicit} complete expanding
gradient K\"ahler-Ricci solitons of cohomogeneity one under the action of the Heisenberg
group in any complex dimension $m>1$.
\begin{theorem}
There exists a complete expanding gradient K\"ahler-Ricci soliton with negative definite Ricci curvature on the manifold $M=\mathcal{H}_{2n+1}\times I$, where $\mathcal{H}_{2n+1}$ is the Heisenberg group in dimension $2n+1$, $n\ge 1$ and $I$ is an open interval. A quotient of $M$ by a discrete subgroup of
$\mathcal{H}_{2n+1}$ admits two ends, and the induced Ricci soliton is asymptotically conical at one end, and asymptotic to a metric of constant holomorphic sectional curvature at the other end. The Ricci flow associated with these Ricci solitons is of type III.
\end{theorem}
\vspace{-.07in}
See Theorem~\ref{sol-Heis2} and formulas \Ref{explicit} for the metric expression,
subsection~\ref{props} for the Ricci curvature, subsection~\ref{asymp} for the asymptotics,
and subsection~\ref{singul} for the statement on the type, which is meant in the sense of \cite{hm1}.
Note that the asymptotics are in analogy with the Ricci soliton found in \cite{r} in dimension three,
which is not known explicitly. Additionally, the occurrence of a non-conical end places these Ricci solitons just outside the class recently classified in \cite{cds}. It would thus be of interest to prove a uniqueness statement characterizing them, thus extending either \cite{cd} or \cite{r}, the latter being based on sectional curvature, see below. We hope to pursue this problem in the future.

We also examine other properties of these solitons. Apart from computing one of the invariants of the
associated Ricci flow, we examine various associated curvatures. Their scalar curvature is (negative)
nonconstant, and bounded. The sectional curvatures are bounded.
We believe that the sectional curvature of a plane $\al$ satisfies the sharp bounds
\[
-\fr2{m+1}<\mathrm{Sec}(\al)<0,
\]
where $m$ is the complex dimension. {\em We give a proof of this claim for $m=2$}.
See section~\ref{singul}. Note that for the Ricci solitons in \cite{r} there are similar
bounds, with the lower one being $-1/4$.


In the case of cohomogeneity one metrics, many compact and complete Ricci solitons, both K\"ahler and non-K\"ahler, were found relatively recently in \cite{d-w, bdw, bdgw, w1, w2}, under compact group actions. Aside from being a non-compact group, the case of the Heisenberg group presented here features K\"ahler-Ricci solitons with a group action having no singular orbits.

Our approach to understanding these Ricci solitons stems from a general framework in which we also show existence of incomplete steady K\"ahler-Ricci solitons, and in dimension four, complete cohomogeneity one metrics under the Euclidean plane group $E(2)$ of a type we call gradient K\"ahler-Ricci
skew-solitons. The definition of the latter includes that of a gradient K\"ahler-Ricci soliton as a
special case. We employ it on manifolds that do not admit gradient K\"ahler-Ricci solitons of the
type we consider in this paper for prosaic reasons: the candidate for a soliton vector field we examine
is not holomorphic. To our knowledge, this natural generalization of a K\"ahler-Ricci soliton does not seem to have been considered in the literature. See \cite{la} for related notions that are discussed in
subsection~\ref{skew-sol0}. Our theorem may be stated as follows.
\begin{theorem}
Complete expanding gradient K\"ahler-Ricci skew-solitons of cohomogeneity one under
an $E(2)$ action with one singular orbit exist on $4$-manifolds admitting diagonal
metrics for such an action.
\end{theorem}
See Theorem~\ref{complete} for a precise statement.

\vspace{.1in}
We now describe the relation of this work to our earlier paper \cite{mr}, and to \cite{a-m2}. The
former paper, which was greatly influenced by work of Dancer and Strachan \cite{d-s1}, explored
diagonal K\"ahler-Einstein metrics in dimension four under a cohomogeneity one action of
unimodular Lie groups of dimension three, the so-called Bianchi type A class. A complete such metric was shown to exist for the Euclidean plane group, while \cite{a-m2} contained a similar result for the three dimensional Heisenberg group. Both of these results were obtained via ODE methods.

Now in \cite{mr} we developed a more general framework within which these $m=2$ cohomogeneity one
metrics are special cases, and another result there gives local existence of K\"ahler Ricci-flat
metrics that are determined by PDEs (more precisely, generalized frame-dependent PDEs), and are not of cohomogeneity one. In the current paper, this framework is extended to all complex dimensions $m>1$. However, while for $m=2$, diagonal cohomogeneity one metrics for all unimodular three dimensional Lie groups are special cases of our general ansatz, its application to cohomogeneity one metrics in dimension $m>2$ is limited. In fact, only the Heisenberg group action fits as a special case when trying to construct expanding solitons of the type we consider. But just as in the $m=2$ case, for $m>2$ one can ask whether more solutions will result from this ansatz with frame-dependent PDEs replacing the above mentioned ODEs. We leave the exploration
of such questions to future work.

Here is one main ingredient of the said framework. As is well-known, integrability of an almost complex structure $J$ is equivalent to the vanishing of the Nijenhuis tensor. In the presence of an almost hermitian metric, we fix some orthogonal decomposition of the tangent bundle into rank $2$ $J$-invariant subbundles. Extending the ideas of \cite{a-m2} and \cite{mr}, we show Nijenhuis vanishing can be given equivalently in the form of $m(m-1)$ so-called shear equations, relating the complex structure with
shear operators, each of which acts on a pair of sections, one from each of two distinct such
subbundles. The utility of this relates to matters of integrability, as one can consider cases in which
many of these shear operators are simply the trivial zero operator. The number of non-zero shear operators in the integrability equations of a given complex structure can be used as a measure of its
complexity. In the scenario described in this paper, at most $m-1$ of the shear equations are
nontrivial. In \cite{mr}, in which $m=2$, there was at most one nontrivial shear equation, and subsequent metrics constructed were simpler in case there were none. We expect this approach to integrability to have other applications.

As the shear operator in a local frame of each rank two subbundle can be given by expressions involving the metric and Lie brackets, we proceed to define our most general ansatz through Lie bracket relations of frame vector fields. Those are designed so that the above $m-1$ shear equations hold, and also so that the metric making this frame orthonormal is K\"ahler. For the latter to hold $m$ additional Lie bracket relations are assumed between frame vector fields forming sections of the same rank $2$-subbundle. An analysis of this additional structure is given
in the appendix.

In Section~\ref{general} we describe the framework (Theorem~\ref{Nij0}) and overall ansatz, and
give the definition of skew-solitons. Section~\ref{gen-sol} gives the generalized PDEs for
skew-solitons in this ansatz, their specialization to ODEs and further conditions required for a
skew-soliton in this ansatz to be an actual K\"ahler-Ricci soliton (Prop.~\ref{skew-sol-lam-ne0}).
In Section~\ref{sec:coho} we show how a certain class of cohomogeneity one metrics fits with our ansatz and derive the soliton and skew-soliton equations in this case (Prop.~\ref{coho-one-sols}).
In Section~\ref{Heisen-sols}, which is our main section, we prove the existence of complete
expanding gradient K\"ahler-Ricci solitons under a Heisenberg group action, show they are complete
(Theorem~\ref{sol-Heis2}), and describe some of their properties, including those related to curvature,
the singularity type for the associated Ricci flow and the asymptotics. In Section~\ref{steady-sols} we
describe steady solitons arising from the specialization of the ansatz to ODEs.
Section~\ref{E2-skews} is devoted to the demonstration of existence of complete K\"ahler-Ricci skew-solitons in complex dimension two under the action of the Euclidean plane group (Theorem~\ref{complete}), following a method employed in \cite{mr}. The appendix contains an analysis of the closedness of the K\"ahler form from the point of view of the framework of Section~\ref{general}.

\section{The general framework}\lb{general}

\subsection{Shear and integrability}\lb{sec:sh}

Let $(M,g,J)$ be an almost hermitian manifold of real dimension $2m$.
Assume a decomposition of the tangent bundle as an orthogonal sum of
rank two $J$-invariant subbundles
\be\lb{decomp}
TM=\bigoplus_{i=1}^m\HH_i. 
\end{equation}
Let $\pi_{\HH_j}:TM\to{\HH_j}$ denote the orthogonal projection.
Consider the operator $\pi_{\HH_j}\circ \n:\Gamma(\HH_j)\times\Gamma(\HH_i)\to\Gamma({\HH_j})$, $j\ne i$, in which we restrict the Levi-Civita covariant derivative $\n$ of $g$ to sections of $\HH_j\times\HH_i$.
Define the  {\em shear operator} of such a pair of subspaces by
\[
S^{ji}\ :=\ \mathrm{Sym}^0[\pi_{\HH_j}\circ \n|_{\HH_j\times\Gamma(\HH_i)}],
\]
where $\mathrm{Sym}^0$ denotes the trace-free symmetric component
with respect to the linear action of $\n X$ on $\HH_j$, for any fixed $X\in\Gamma(\HH_i)$.

See \cite{a-m} for background on the relation to the shear operator in general relativity.

Our purpose here is to give a condition equivalent to the integrability of $J$ in terms of
shear operators.
\begin{thm}\lb{Nij0}
Given the above set-up, the almost complex structure $J$ is integrable 
if and only if for all $i=1,\ldots m$ and any $j\ne i$
\be\lb{Nij}
J S^{ji}_vw=S^{ji}_v{Jw},\quad \text{for all $w\in\Gamma(\HH_i)$ and $v\in\HH_j$.}
\end{equation}
\end{thm}
{\em Proof sketch.} The issue is local, so consider a frame domain for an orthonormal
frame $\{e_i\}$ such that locally $\HH_i=\mathrm{span}(e_{2i-1},e_{2i})$.
It is enough to show the Nijenhuis tensor $N$ vanishes on pairs of frame fields.
Its defining formula clearly shows that it vanishes if both frame fields are in $\HH_i$. Because
$N(a,b) = JN(a,Jb)$ and $N$ is antisymmetric, if $j\ne i$ it is enough to check under what
condition it vanishes on a pair $e_{2i-1}$, $e_{2j-1}$. That calculation proceeds similarly to
\cite[Theorem 1]{a-m}, and will be omitted.
\qed

We point out that the calculation just alluded to relies on the following expression of the matrix
corresponding to the shear operator, evaluated on a vector field $X\in\Gamma(H_i)$ in a local oriented orthonormal frame $\{e_{2j-1},e_{2j}\}$ on $\HH_j$.
\[
[S^{ji}X]_{e_{2j-1},e_{2j}}=
\begin{bmatrix}
        - \sigma_1 & \sigma_2\\
        \sigma_2 & \sigma_1\\
      \end{bmatrix},
\]
with {\em shear coefficients}:
\be\lb{sh-coef}
\begin{aligned}
2\sigma_1\ &:=\
 \ \ g([X,e_{2j-1}],e_{2j-1})-g([X,e_{2j}],e_{2j}),\\
2\sigma_2\ &:=\
 -g([X,e_{2j-1}],e_{2j}) - g([X,e_{2j}],e_{2j-1}).
\end{aligned}
\end{equation}

Note that to check integrability via Theorem~\ref{Nij0}, one has to verify $m(m-1)$
operator equations. Simpler special cases occur if some of the shear operators $S^{ji}$
happen to be the zero operator, in which case the corresponding equation~\Ref{Nij} holds
automatically. In the next subsection we will define an ansatz for K\"ahler metrics for which
$S^{ji}=0$ for all pairs $ji$ except for those of the form $ji_0$ for {\em one particular index}
$i_0$. Thus there will only be $m-1$ non-trivial equations involving shear operators.
All metrics discussed in this paper will be special cases of this ansatz.


\subsection{K\"ahler frame system}\lb{construct}

Let $(M,g)$ be a Riemannian manifold of dimension $2m=2n+2$ admitting an orthonormal frame $\{e_j\}_{j=1}^{2m}=\kk, \tT, \{\xx_i, \yy_i\}_{i=1}^n$,
defined over an open $U\subset M$, which satisfies the following Lie bracket relations:
$[\xx_i,\xx_j]$, $[\xx_i,\yy_j]$, and $[\yy_i,\yy_j]$ vanish whenever $i\ne j$, while
\begin{align}
&[\kk,\tT]=L(\kk+\tT),\qquad &&[\xx_i,\yy_i]=N_i(\kk+\tT),\lb{brack1}\\
&[\kk,\xx_i]=A_i\xx_i+B_i\yy_i,\qquad  &&[\kk,\yy_i]=C_i\xx_i+D_i\yy_i,\lb{brack2}\\
&[\tT,\xx_i]=E_i\xx_i+F_i\yy_i,\qquad  &&[\tT,\yy_i]=G_i\xx_i+H_i\yy_i,\lb{brack3}
\end{align}
for smooth functions $A_i, B_i, C_i, D_i, E_i, F_i, G_i, H_i, L, N_i$ on $U$ such that
\begin{align}
A_i&-D_i=F_i+G_i,\qquad B_i+C_i=H_i-E_i,\lb{rels1}\\
N_i&=A_i+D_i=-(E_i+H_i)\lb{rels2}
\end{align}
for $i=1,\ldots n$.
Define an almost complex structure
$J=J_{g,\{e_j\}}$ by linearly extending the relations $J\kk=\tT$, $J\tT=-\kk$, $J\xx_i=\yy_i$
and $J\yy_i=-\xx_i$, $i=1,\ldots n$.

Another proof of the following proposition, which employs the decomposition
\Ref{decomp} in a different way, is outlined in the appendix.
\begin{prop}\lb{kah}
(M,g,J) defined as above gives a K\"ahler structure on $U$.
\end{prop}

\begin{proof}
$J$ clearly makes $g$ into an almost hermitian metric.
To see that $J$ is integrable, we verify the conditions of Theorem~\ref{Nij0}.
Changing slightly the indexing of the previous subsection,
let $\HH_0=\mathrm{span}(\kk,\tT)$, $\HH_i=\mathrm{span}(\xx_i, \yy_i)$, $i=1,\ldots n$,
with the frame indexing beginning with $\kk=e_{\scriptscriptstyle{-1}}$, $\tT=e_{\scriptscriptstyle{0}}$.
In view of \Ref{sh-coef}, the vanishing Lie bracket relations clearly show that
the shear operators $S^{ji}=0$ if $i\ne 0$, $j\ne 0$ and $i\ne j$. Additionally
$S^{0i}=0$ for all $i=1,\ldots, m-1$ as well, since \Ref{sh-coef}, with $X=\xx_i$ or $X=\yy_i$
and $e_{2j-1}=\kk$, $e_{2j}=\tT$ yields zero shear coefficients for $X$ in view of \Ref{brack2}-\Ref{brack3} and the orthonormality of our frame. Finally, to check \Ref{Nij} for
$S^{j0}$ we note that in terms of shear coefficients this equation takes the form
$\sig_1^\tT=\sig_2^\kk$, $\sig_2^\tT=-\sig_1^\kk$,
and these hold as they are equivalent, by \Ref{sh-coef}, to the assumed relations \Ref{rels1}.

To show that $g$ is K\"ahler,
define a connection on $U$ by first setting
\be\lb{conn}
\n_\kk\kk=-L\tT,\quad  \n_{\xx_i}\xx_i=A_i\kk+E_i\tT,\quad
\n_{\xx_i}\kk=-A_i\xx_i+E_i\yy_i,\quad \n_{\xx_i}\xx_j=0, i\ne j,
\end{equation}
and then having all other covariant derivative expressions on frame fields
determined by the requirement that $\n$ be torsion-free and make $J$ parallel
(here the definition of $J$ and relations \Ref{brack1}-\Ref{rels2} are used
repeatedly).
It is easily checked that $\n$, thus defined, is compatible with the metric $g$, so that
it is its Levi-Civita connection and hence $J$ is $g$-parallel. This completes the proof.
\end{proof}

One additional object comes for free with the K\"ahler structure in Proposition~\ref{kah}.
Namely, the Lie bracket relations \Ref{brack1}-\Ref{brack3} imply that the distribution spanned by
$\kk+\tT$, $\xx_i$ and $\yy_i$, $i=1,\ldots n$ is integrable. Since this distribution is
orthogonal to $\kk-\tT$, while the latter vector field has constant length and is easily seen to
have geodesic flow, it follows that it is locally a gradient (cf. \cite[Cor. 12.33]{onel}). Thus, there exists
a smooth function $\ta$ defined in some open set $V\subset U$, such that
\be\lb{grad}
\kk-\tT=\n\ta.
\end{equation}
As in \cite{mr}, in searching for distinguished metrics we will assume that
the functions $A_i,\ldots H_i, N_i$, $i=1,\ldots n$ and $L$  are each a composition with $\ta$ of a smooth real-valued function defined on the image of $\ta$, and will abuse notation by denoting the latter functions
by the same respective letters as the former. In the case of non-Ricci-flat K\"ahler-Einstein
metrics of dimension four that fit the ansatz of Prop.~\ref{kah}, we have shown in \cite{mr} that
this assumption always holds necessarily. We will not attempt to extend that result in this paper.

\subsection{The Ricci form}

The Ricci form of the K\"ahler metric $g$ in Proposition~\ref{kah} is computed
as follows. Denote by $w_{\scriptscriptstyle{0}}=\kk-i\tT$, $w_i=\xx_i-i\yy_i$, $i=1,\ldots n$ the
corresponding complex-valued frame, and compute the complex valued $1$-forms
$\Gamma_i^j$ for which $\n w_i=\Gamma_i^j\otimes w_j$, where here $\n$
denotes the obvious complexification of the Levi-Civita connection of $g$ and the
summation convention was used. The formulas are deduced by computing the components
$\n_{e_\ell} w_i$, where $e_\ell$ stands for one of the frame fields, using the covariant
derivative frame formulas for the Levi-Civita connection $\n$, given in the proof of Proposition~\ref{kah}. The $1$-forms $\Gamma_{\scriptscriptstyle{0}}^{\scriptscriptstyle{0}}$, $\Gamma_j^j$, $j=1,\ldots n$ resulting
from this calculation are given by
\[
\Gamma_{\scriptscriptstyle{0}}^{\scriptscriptstyle{0}}=-iL(\hat\kk+\hat\tT),\qquad \Gamma_j^j=-i(C_j-H_j)\hat\kk-i(A_j-F_j)\hat\tT,\quad j=1\ldots n,
\]
where the hatted quantities denote  the non-metrically-dual coframe of $\{e_\ell\}$.
Citing, for example, Lemma 4.2 in \cite{dr-ma}, the Ricci form of $g$ is given by
\begin{align}
\rho&=i(d\Gamma_{\scriptscriptstyle{0}}^{\scriptscriptstyle{0}}+{\textstyle\sum_{j=1}^n} d\Gamma_j^j)=L(d\hat\kk+d\hat\tT)+\Big({\textstyle\sum_{j=1}^n}(C_j-H_j)\Big)\,d\hat\kk\nonumber\\
&+\Big({\textstyle\sum_{j=1}^n}(A_j-F_j)\Big)\,d\hat\tT
+dL\we(\hat\kk+\hat\tT)
+\Big({\textstyle\sum_{j=1}^n} d(C_j-H_j)\Big)\we\hat\kk\lb{ric}\\
&+\Big({\textstyle\sum_{j=1}^n} d(A_j-F_j)\Big)\we\hat\tT.\nonumber
\end{align}

We now wish to write the Ricci components in our frame.
Applying to our coframe the formula
$d\eta(a,b)=d_a(\eta(b))-d_b(\eta(a))-\eta([a,b])$,
valid for any smooth $1$-form $\eta$, we have
\be\lb{d-frame}
\begin{aligned}
&d\hat\kk(\kk,\tT)=-L=d\hat\tT(\kk,\tT),\\
&d\hat\kk(\xx_i,\yy_i)=-\hat\kk([\xx_i,\yy_i])=-\hat\kk (N_i(\kk+\tT))=-N_i=d\hat\tT(\xx_i,\yy_i),\\
&d\hat\kk(\kk,\xx_i)=d\hat\kk(\kk,\yy_i)=
d\hat\kk(\tT,\xx_i)=d\hat\kk(\tT,\yy_i)=0,\\
&d\hat\tT(\kk,\xx_i)=d\hat\tT(\kk,\yy_i)=d\hat\tT(\tT,\xx_i)=d\hat\tT(\tT,\yy_i)=0,
\end{aligned}
\end{equation}
whereas $d\hat\kk$, $d\hat\tT$ vanish on pairs taken from $\xx_i$, $\yy_i$, $\xx_j$, $\yy_j$
for $i\ne j$.

Using this in \Ref{ric}
we have for $i=1,\ldots n$
\be\lb{Ric-frame}
\begin{aligned}
&\rho(\xx_i,\yy_i)=-N_i(2L+{\textstyle\sum_{j=1}^n\nolimits}(C_j-H_j+A_j-F_j)),\\[2pt]
&\rho(\kk,\tT)=-L(2L+{\textstyle\sum_{j=1}^n\nolimits}(C_j-H_j+A_j-F_j))\\[1pt]
&\qquad\quad\ \, \ \, +d_{\kk-\tT}L-{\textstyle\sum_{j=1}^n\nolimits}(d_\tT(C_j-H_j)+d_\kk(A_j-F_j)),\\[2pt]
&\rho(\kk,\xx_i)=-d_{\xx_i}(L+{\textstyle\sum_{j=1}^n\nolimits}(C_j-H_j)),\\[2pt]
&\rho(\kk,\yy_i)=-d_{\yy_i}(L+{\textstyle\sum_{j=1}^n\nolimits}(C_j-H_j)),\\[2pt]
&\rho(\tT,\xx_i)=-d_{\xx_i}(L+{\textstyle\sum_{j=1}^n\nolimits}(A_j-F_j)),\\[2pt]
&\rho(\tT,\yy_i)=-d_{\yy_i}(L+{\textstyle\sum_{j=1}^n\nolimits}(A_j-F_j)),\\[2pt]
&\rho(\xx_i,\xx_j)=0,\qquad \rho(\xx_i,\yy_j)=0,\qquad \rho(\yy_i,\yy_j)=0\ \  \text{   for $i\ne j$},
\end{aligned}
\end{equation}
where $d_{e_\ell}$ denotes the directional derivative with respect to $e_\ell$.

\subsection{The Ricci soliton equation and its two generalizations}\lb{skew-sol0}

We discuss here two soliton-type equations,
one of which will be examined later for $(M,g)$ as in Proposition~\ref{kah}.

Consider on a K\"ahler manifold an equation of the form
\be\lb{sol}
\rho+\tfrac 12\mathcal{L}_X\om=\lam\,\om
\end{equation}
where $\rho$ is the Ricci form, $\om$ the K\"ahler form, $\lam$ is a constant and
$\mathcal{L}_X$ is the Lie derivative with respect to a {\em smooth} vector field $X$.
This is the Chern-Ricci soliton equation \cite{la}, even on a K\"ahler manifold, as
$X$ is just smooth and possibly not holomorphic.
Computing this Lie derivative term on any K\"ahler manifold, we have,
\begin{multline}\lb{Chern}
\mathcal{L}_X\om(a,b)=(d\imath_X\om)(a,b)=(\n_a(\imath_X\om))(b)-(\n_b(\imath_X\om))(a)\\
=\om(\n_aX,b)-\om(\n_bX,a)=g(\n_a(JX),b)-g(\n_b(JX),a),
\end{multline}
as $\om$ and $J$ are parallel.
Note that this expression is generally different from the skew-symmetric expression
\be\lb{skew}
g(\n_{Ja}X,b)-g(\n_{Jb}X,a)
\end{equation}
However, the two expressions are equal if $X$ is holomorphic, since then
\begin{multline*}
0=(\mathcal{L}_XJ)(a)=\mathcal{L}_X(Ja)-J\mathcal{L}_Xa=[X,Ja]-J[X,a]\\
\n_X(Ja)-\n_{Ja}X-J\n_Xa+J\n_aX=J\n_aX-\n_{Ja}X,
\end{multline*}
and this observation also shows that \Ref{Chern} and \Ref{skew}
are indeed generally not equal if $X$ is smooth but not holomorphic, as $\mathcal{L}_XJ$
is skew-adjoint rather than self-adjoint for a hermitian metric.

The condition that $X$ is holomorphic and not just smooth turns \Ref{sol}
into the standard Ricci soliton equation. Now if $X=\n f$ is a gradient
of a smooth function $f$, then $X$ is holomorphic exactly when $JX$ is
a Killing field (cf. \cite[Lemma 5.2]{dr-ma}). In the following we will
use the latter criteria to examine gradient Ricci solitons. However, we will
also look at cases where $X=\n f$ is only smooth, and between the two
generalizations \Ref{Chern} and \Ref{skew} of the gradient Ricci soliton
condition, we will be examining the latter, i.e. the equation
\be\lb{skew-sol}
\rho(a,b)+\tfrac 12(\n df(Ja,b)-\n df(Jb,a))=\lam\,\om(a,b),
\end{equation}
which corresponds to \Ref{skew} since $\n df(Ja,b)=g(\n_{Ja}\n f,b)$. We will
call pairs $(g,f)$ satisfying \Ref{skew-sol} {\em gradient K\"ahler-Ricci skew-solitons}.

\section{The frame-dependent form of the soliton equation}\lb{gen-sol}

We now consider (M,g,J) as in Proposition~\ref{kah}.
Employing the formula
\[
\n df(e_i,e_j)=d_{e_i}d_{e_j}f-df(\n_{e_i}e_j)
\]
and noting that our K\"ahler form is just $\om=\hat\kk\we\hat\tT+\sum_{i=1}^n\hat\xx_i\we\hat\yy_i$,
one calculates, using also \Ref{Ric-frame} and the covariant derivative formulas stemming from
\Ref{conn}, that the skew-soliton equation \Ref{skew-sol} is equivalent to the following system of
frame-dependent PDEs, for each $i=1,\ldots n$
\be\lb{sol-eqns}
\begin{aligned}
&-N_i(2L+{\textstyle\sum_{j=1}^n}(C_j-H_j+A_j-F_j))+{\textstyle\fr12} (d_{\xx_i}^2f+d_{\yy_i}^2f-N_i(d_\kk f-d_\tT f))=\lam,\\[2pt]
&-L(2L+{\textstyle\sum_{j=1}^n}(C_j-H_j+A_j-F_j))\\[1pt]
&\ \ \ \ \ \ \ \ \ \ \ +d_{\kk-\tT}L
-{\textstyle\sum_{j=1}^n}(d_\tT(C_j-H_j)+d_\kk(A_j-F_j))\\[1pt]
&\ \ \ \ \ \ \ \ \ \ \ \ \ \ \ \ \ \ \ \ \ \ \ \ \ \ \ \ \ \ +{\textstyle\fr12}(d_\kk^2f+d_\tT^2f-L(d_\kk f-d_\tT f))=\lam,\\[2pt]
&-d_{\xx_i}(L+{\textstyle\sum_{j=1}^n}(C_j-H_j))+{\textstyle\fr12}(d_{\xx_i} d_\tT f-d_\kk d_{\yy_i} f-B_id_{\xx_i} f+A_id_{\yy_i} f)=0,\\[2pt]
&-d{\textstyle_{\yy_i}(L+\sum_{j=1}^n}(C_j-H_j))+{\textstyle\fr12}(d_{\yy_i} d_\tT f+d_\kk d_{\xx_i} f-D_id_{\xx_i} f+C_id_{\yy_i} f)=0,\\[2pt]
&-d_{\xx_i}(L+{\textstyle\sum_{j=1}^n}(A_j-F_j))+{\textstyle\fr12}(-d_{\xx_i} d_\tT f-d_\tT d_{\yy_i} f-F_id_{\xx_i} f+E_id_{\yy_i} f)=0,\\[2pt]
&-d_{\yy_i}(L+{\textstyle\sum_{j=1}^n}(A_j-F_j))+{\textstyle\fr12}(-d_{\yy_i} d_\kk f+d_\tT d_{\xx_i} f-H_id_{\xx_i} f+G_id_{\yy_i} f)=0,\\[2pt]
&\ \ \ \, d_{\yy_i}d_{\xx_j}f-d_{\yy_j}d_{\xx_i}f=0,\qquad d_{\yy_i}d_{\yy_j}f+d_{\xx_j}d_{\xx_i}f=0
\text{ for $i\ne j$.}
\end{aligned}
\end{equation}

As for the additional conditions for this to be K\"ahler-Ricci soliton, we will compute them
later only for special choices of $f$.

\subsection{The $\ta$-dependent case}\lb{ta-dep}

We now make the assumption that the skew-soliton potential $f$ is a composition of a
function on the range of $\ta$, with $\ta$. By abuse of notation we write.
\[
f=f(\ta).
\]
Additionally, as mentioned earlier, we assume the functions $A_i,\ldots,H_i, N_i, L$ are also
such composites. We call this setting the {\em $\ta$-dependent case}.

Now $d_{\xx_i}\ta=g(\xx_i,\kk-\tT)=0$ and similarly
$d_{\yy_i}\ta=0$, $i=1\ldots n$, while $d_\kk\ta=1$ and $d_\tT\ta=-1$.
It follows that under these assumptions the equations represented in the last
five lines of \Ref{sol-eqns} are satisfied trivially,
while those in the first two lines become
\begin{align}
&-N_i(2L+{\textstyle\sum_{j=1}^n}(C_j-H_j+A_j-F_j))-N_if'=\lam,\lb{1}\\[2pt]
&-L(2L+{\textstyle\sum_{j=1}^n}(C_j-H_j+A_j-F_j))\nonumber\\[1pt]
&\ \ \ \ \ \ \ \ \ \ \ \ \ \ \ \ +2L'+{\textstyle\sum_{j=1}^n}(C_j'-H_j'+A_j'-F_j')+f''-Lf'=\lam,\lb{2}
\end{align}
where the prime denotes differentiation with respect to (a variable on the image of) $\ta$.

Consider the case $\lam\ne 0$. In that case, from \Ref{1},
$2L+\sum_{j=1}^n(C_j-H_j+A_j-F_j)-f'$ is nowhere zero, and isolating $N_i$ from that equation,
we see $N_i$, which is also nonzero, is independent of $i$. Hence in analyzing this
case we denote
\[
N=N_i,\qquad i=1\ldots,n.
\]
Writing \Ref{1} in the form
\be\lb{1'}
2L+{\textstyle\sum_{j=1}^n}(C_j-H_j+A_j-F_j)=-\fr{\lam}N-f'
\end{equation}
we see that \Ref{2} is just
\[
L\Big(\fr{\lam}N+f'\Big)-\Big(\fr{\lam}N+f'\Big)'+f''-Lf'=\lam
\]
which simplifies to
\be\lb{3}
\lam\Big(\fr LN+\fr {N'}{N^2}-1\Big)=f''-Lf'-f''+Lf'=0
\end{equation}
However, we will now see that this equation is in fact an identity, so that
we only need to consider \Ref{1'}.

From relations \Ref{d-frame}, we have, returning momentarily to the notation $N_i$,
\[
d\hat{\kk}=-{\textstyle\sum_{i=1}^n}N_i\hat\xx_i\we\hat\yy_i-L\hat\kk\we\hat\tT.\\
\]
Expanding $d^2\kk=0$ and substituting
\[
dN_i=d_\kk N_i\,\hat\kk+d_\tT N_i\,\hat\tT+{\textstyle\sum_{i=1}^n}(d_{\xx_i}\! N_i\,\hat\xx_i+d_{\yy_i}\! N_i\,\hat\yy_i),
\]
we get a complicated expression, where the coefficients of $\hat\kk\we\xx_i\we\yy_i$
and $\hat\tT\we\xx_i\we\yy_i$ are, respectively
\begin{align*}
-d_\kk N_i+N_iA_i+N_iD_i-LN_i=0,\\
-d_\tT N_i+N_iE_i+N_iH_i+LN_i=0,
\end{align*}
Subtracting the second of these from the first, then using $d_{\kk-\tT}=2\partial_\ta$
along with the relations \Ref{rels2} and dividing by $-2N_i^2$, gives \Ref{3}.


\vspace{.2in}
We now check the extra conditions needed for the skew-soliton to be a
K\"ahler-Ricci soliton, by examining when $X=J\nabla f$ is a Killing vector field.
As $f=f(\tau)$, we have $\nabla f=f'(\tau)(\kk-\tT)$ so $J\nabla f=f'(\tau)(\kk+\tT)$.
Computing
\[
(\mathcal{L}_Xg)(a,b)=g([a,X],b)+g([b,X],a)
\]
on our frame, we find
\begin{align*}
{\textstyle\frac{1}{2}}\mathcal{L}_{X}g&=(f''(\tau)+Lf'(\tau))(\kf^2-\tf^2)\\
&-f'(\tau)\Big({\textstyle\sum_{i=1}^n}
\big[(A_i+E_i)\xf_i^2\\
&+(B_i+C_i+F_i+G_i)\xf_i\odot\yf_i+(D_i+H_i)\yf_i^2\big]\Big).
\end{align*}
Therefore, at points where $f'(\ta)\ne 0$, $J\nabla f$ is Killing when
\begin{align}
\label{JdfK1}f''(\tau)+Lf'(\tau)&=0,\\
\label{JdfK2}A_i+E_i&=0,\\
\label{JdfK3}D_i+H_i&=0,\\
\label{JdfK4}B_i+C_i+F_i+G_i&=0,
\end{align}
for $i=1,\ldots,n$.

We summarize the results of this subsection.
\begin{prop}\lb{skew-sol-lam-ne0}
Let $(M,g,J)$ be a K\"ahler manifold as in Proposition~\ref{kah},
with $A_i,\ldots,H_i,N_i$, $i=1,\ldots,n$ and $L$ functions of $\ta$.
Then $g$ is a gradient Ricci skew-soliton with soliton potential $f=f(\ta)$
and nonzero soliton constant $\lam$ if and only if all functions $N_i$
are equal to a single function $N$ and equation~\Ref{1'} holds. Furthermore,
$(g,f)$ is a gradient Ricci soliton if and only if additionally
equations~\Ref{JdfK1}-\Ref{JdfK4} hold for all $i=1\ldots n$ away from
critical points of $f(\ta)$ and every such critical point is degenerate.
\end{prop}

\section{The cohomogeneity one subclass}\lb{sec:coho}
In this section we begin a study of a fairly explicit class of examples. We first discuss
cohomogeneity one metrics admitting a frame satisfying conditions \Ref{brack1}-\Ref{rels2}.
Our discussion follows \cite{mr} closely.

Let $(M,g)$ be a Riemannian manifold of dimension $2m=2n+2$
admitting a proper isometric action by a Lie group $\mathcal{G}$ with cohomogeneity one.
Then there is a subgroup $\mathcal{H}<\mathcal{G}$ so that $\mathcal{G}/\mathcal{H}$ is the $2n+1$-dimensional principal orbit type.
Let $p\in M$ be a point with isotropy group $\mathcal{K}$ satisfying $\mathcal{H}\le\mathcal{K}<\mathcal{G}$. Then the orbit $\mathcal{G}\cdot p$ through $p$ is isomorphic to $\mathcal{G}/\mathcal{K}$. Consider the following complementary facts. For the principal $\mathcal{K}$-bundle $\mathcal{G}\to\mathcal{G}/\mathcal{K}$,
one has an associated bundle $\mathcal{G}\times_{\mathcal{K}}\nu_p$, where $\nu_p$ is
the normal space to the orbit at $p$. The differential of the action mapping identifies this bundle
with the full normal bundle $\nu$ to the orbit. On the other hand, the normal exponential map $\exp_p^\perp$ at $p$ sends an $\varepsilon$-disk in $\nu_p$ to a
slice for the action of $\mathcal{G}$
\[ S' = \{\exp_p(rX)\ |\ 0\le r < \varepsilon, |X|=1, X\perp \mathcal{G}\cdot p \}. \]
and induces, by the tubular neighborhood theorem, a map from a neighborhood of the
zero section in $\nu$ to a neighborhood of the orbit. Putting these facts together
we obtain an equivariant diffeomorphism,
\[ \mathcal{G}\times_\mathcal{K} D^{\ell+1}\cong\mathcal{G}\cdot S', \]
where $\ell=\mathrm{dim}\,\mathcal{K}-\mathrm{dim}\,\mathcal{H}$ (cf. \cite[Section 5.6]{pp}).

The isotropy action of $\mathcal{K}$ in fact preserves length, so on $S'$, we see that the spheres
\[ S_r = \{\exp_p(rX)\ |\ |X|=1, X\perp \mathcal{G}\cdot p \} \]
are preserved by the induced action of $\mathcal{K}$.
Since points on one of these spheres have isotropy type $\mathcal{H}$,
we must have \[ \mathcal{K}/\mathcal{H} \cong \mathbb{S}^\ell. \]

Regarding the metric as residing on $\mathcal{G}\times_\mathcal{K} D^{\ell+1}$, it can be written
in the form
\begin{equation}\label{equivmet} dr^2+g_r.\end{equation}
We will now assume that $\mathcal{G}$ has dimension $2n+1$ with $\mathcal{H}$ a discrete
principal stabilizer. In the case of a unimodular group, we consider the special case of a
diagonal metric, in the form inspired by the four-dimensional case appearing,  for example, in \cite{d-s1}.
Namely, we write
\begin{equation}\label{bianchiAmet} g = \Big(\prod_{i=1}^{n}(a_ib_i)^2\Big)c^2\,dt^2+c^2\zeta^2
+\sum_{i=1}^{n}(a_i^2\sig_i^2+b_i^2\rho_i^2), \end{equation}
for functions $a_i$, $b_i$ and $c$ of $t$ and left invariant $1$-forms $\sig_i$, $\rho_i$ and $\zeta$.

In terms of the frame $\partial_t,Z,\{X_i\}_{i=1}^{n},\{Y_i\}_{i=1}^{n}$ non-metrically dual, respectively, to the coframe $dt,\zeta,\{\sig_i\}_{i=1}^{n},\{\rho_i\}_{i=1}^{n}$, we restrict the
possible groups $\mathcal{G}$ by requiring repeated Bianchi type A structure constants:
\begin{equation}\label{X123}
\begin{split}
[\partial_t,X_i]&=[\partial_t,Y_i]=[\partial_t,Z]=0,\quad i=1,\ldots n,\\[3pt]
[X_i,Y_i]&=-{\mathlarger\Gamma}_{ii}^{z} Z,\quad i=1,\ldots n,\\[3pt]
[Y_i,Z\,]&=-\mathlarger{\Gamma}_{iz}^{i} X_i,\quad i=1,\ldots n,\\[3pt]
[Z,X_i\,]&=-\mathlarger{\Gamma}_{zi}^{i} Y_i,\quad i=1,\ldots n,\\[3pt]
[X_i,X_j]&=0,\quad [X_i,Y_j]=0 \text{ for $i\ne j$,\quad  $i,j=1,\ldots n$},
\end{split}
\end{equation}
where $\Gamma_{ab}^c$ stand for constants and all other structure constants vanishing or
determined by the Lie algebra requirements. Note that the summation convention is {\em not} used here.
For $n>1$ these are very stringent conditions: the Jacobi identity for $X_i$, $Y_i$, $X_j$ and for
$X_i$, $Y_i$, $Y_j$ with $j\ne i$ implies that either ${\mathlarger\Gamma}_{ii}^{z}=0$ or both ${\mathlarger\Gamma}_{jz}^{z}$ and ${\mathlarger\Gamma}_{zj}^{z}$ vanish for all $j\ne i$.


The almost complex structure will be determined in this frame by
\begin{equation}\label{complexJ}
J\partial_t = \Big(\prod_{i=1}^na_i b_i\Big)Z \quad\text{and}\quad JX_i=\frac{a_i}{b_i}Y_i,\quad i=1\ldots n.
\end{equation}

In the next subsection we 
show how this model fits within the framework of Section~\ref{general}.

\subsection{A frame as in Section~\ref{construct}}\lb{bian}

In this subsection we show how the metric $g$ of the previous subsection
gives rise to data satisfying \Ref{brack1}-\Ref{rels2}, and write down
equations equivalent to \Ref{sol-eqns} for it.

Set $\al=(\prod_{i=1}^{n}(a_ib_i))c$, and consider the orthonormal frame and dual coframe
\begin{align*}
&\mathbf{k} = \left(Z/c+\partial_t/\al \right)/\sqrt{2}, && \hat{\mathbf{k}} = (c\,\zeta+\al\,dt)/\sqrt{2}, \\
&\mathbf{t} = \left(Z/c-\partial_t/\al \right)/\sqrt{2}, && \hat{\mathbf{t}} = (c\,\zeta-\al\,dt)/\sqrt{2}, \\
&\mathbf{x}_i = X_i/a_i, && \hat{\mathbf{x}}_i = a_i\sig_i,\quad i=1\ldots n\\
&\mathbf{y}_i = Y_i/b_i,     && \hat{\mathbf{y}}_i = b_i\rho_i,\quad i=1\ldots n.
\end{align*}

It can be checked that this frame satisfies \Ref{brack1}-\Ref{brack3}
for the functions
\be\lb{L-N}
\begin{aligned}
A_i&=-E_i=-\frac{a_i'}{\sqrt{2}\al a_i}=-\frac{1}{a_i}\frac{da_i}{d\tau}, & B_i &=F_i= -\frac1{\sqrt{2}}\mathlarger{\Gamma}_{zi}^i \fr{b_i}{ca_i}, \\
D_i&=-H_i=-\frac{b_i'}{\sqrt{2}\al b_i}=-\frac{1}{b_i}\frac{db_i}{d\tau},
& C_i&=G_i= \frac1{\sqrt{2}}\mathlarger{\Gamma}_{iz}^i \fr{a_i}{cb_i},\\
L&=-\frac{c'}{\sqrt{2}\al c}=-\frac{1}{c}\frac{dc}{d\tau},
& N_i &= -\frac1{\sqrt{2}}{\mathlarger\Gamma}_{ii}^{z}
\fr{c}{a_ib_i}.
\end{aligned}
\end{equation}

Here the prime denotes differentiation with respect to $t$, and
\[ \hat{\mathbf{k}}-\hat{\mathbf{t}}=d\tau=\sqrt{2}\al\,dt,\]
so that
\[ \frac{d}{d\tau}=\frac{1}{\sqrt{2}\al}\frac{d}{dt}. \]

\vspace{0.1in}
We note that for each $i=1\ldots n$, requiring the functions $A_i,\ldots,H_i,N_i$ and $L$, to fulfill
the four relations in \Ref{rels1}-\Ref{rels2} implies that in this model the K\"ahler condition
imposes only two additional relations here, say $A_i+D_i=N_i$ and $B_i+C_i=H_i-E_i$,
giving for each $i=1\ldots n$
\begin{align}
\fr{a_i'}{a_i}+\fr{b'_i}{b_i}&=\mathlarger{\Gamma}_{ii}^z\fr{\al c}{a_ib_i},\lb{K3}\\
\fr{b_i'}{b_i}-\fr{a_i'}{a_i}&=\Big(\mathlarger{\Gamma}_{iz}^i\fr{a_i}{b_i}
-\mathlarger{\Gamma}_{zi}^i\fr{b_i}{a_i}\Big)\fr{\al}c.\lb{another}
\end{align}

We now impose the skew-soliton and soliton conditions of Proposition~\ref{skew-sol-lam-ne0}.
We assume the soliton constant is nonzero and the soliton potential depends on $\ta$.
Recall this implies all functions $N_i$ equal the same function $N$, i.e.
\be\lb{uniform}
\text{$\mathlarger{\Gamma}_{ii}^z(a_ib_i)^{-1}$ is independent of $i$.}
\end{equation}
Then equation \Ref{1'} can be written with the help of \Ref{K3} in the form
\be\lb{1''}
\fr{c'}c=\al\Big[\sum_{j=1}^n\Big[\fr 1{2a_jb_j}\Big(\fr1{c}\big(\mathlarger{\Gamma}_{jz}^ja_j^2+\mathlarger{\Gamma}_{zj}^jb_j^2\big)
-\mathlarger{\Gamma}_{jj}^zc\Big)\Big]-\fr{\lam}{\mathlarger{\Gamma}_{ii}^z}\fr{a_ib_i}{c}\Big]+\fr12f'
\end{equation}
We now consider the additional conditions for a Ricci soliton, that must hold
away from critical points of $f(\ta)$. Two of them, \Ref{JdfK2} and \Ref{JdfK3}
always hold in our model. Condition \Ref{JdfK4} simplifies to
\be\lb{Gam-ab}
\mathlarger{\Gamma}_{iz}^ia_i^2-\mathlarger{\Gamma}_{zi}^ib_i^2=0,\quad i=1,\ldots n,
\end{equation}
so only one of the terms in this difference is retained in \Ref{1''},
but we also see from \Ref{another} that
\be\lb{a-eq-b}
\text{$b_i=\ell_i a_i$ for a constant $\ell_i>0$ satisfying
$\mathlarger{\Gamma}_{iz}^i-\ell_i^2\mathlarger{\Gamma}_{zi}^i=0$, $i=1,\ldots n$.}
\end{equation}
Finally, condition \Ref{JdfK1} simplifies to $f''-\fr{(\al c)'}{\al c}f'=0$,
which is equivalent to
\be\lb{f-pr} \text{$f'=k\al c$ for a constant $k$.}\end{equation}
Substituting the consequences of these two conditions in \Ref{1''}, we arrive at the following equation
for a gradient K\"ahler-Ricci soliton in this model:
\be\lb{1'''}
\fr{c'}c=\al\Big[\sum_{j=1}^n\Big[\fr 1{2\ell_j a_j^2}\Big(\fr2{c}\mathlarger{\Gamma}_{jz}^ja_j^2
-\mathlarger{\Gamma}_{jj}^zc\Big)\Big]-\fr{\lam}{\mathlarger{\Gamma}_{ii}^z}\fr{\ell_i a_i^2}{c}+\fr12kc\Big].
\end{equation}
To this we can add the modification of equation \Ref{K3}:
\be\lb{K3'}
\fr{a_i'}{a_i}=\fr{\mathlarger{\Gamma}_{ii}^z}2\fr{\al c}{\ell_ia_i^2},\quad i=1,\ldots n.
\end{equation}
We summarize
\begin{prop}\lb{coho-one-sols}
Let $(M^{2n+2},g)$ be a cohomogeneity one manifold under a proper action of
a Lie group $\mathcal{G}$ whose Lie algebra relations in a given frame are
as in the last four lines of \Ref{X123}, while $g$ has the
form~\Ref{bianchiAmet}. If $\ta=\ta(t)$ is a solution of
$\ta'(t)=\sqrt{2}\big(\prod_{i=1}^n a_ib_i\big)c$, then $(g,f(\ta))$ is a K\"ahler-Ricci
skew-soliton with nonzero soliton constant if \Ref{K3}-\Ref{1''}
hold. If additionally \Ref{Gam-ab} holds, and \Ref{f-pr} also holds for a constant $k$ and
$f=f(\ta(t))$, then \Ref{a-eq-b}, \Ref{1'''} and \Ref{K3'} also hold, and
if, furthermore, every critical point of $f(\ta)$ is degenerate, then
$g$ is a K\"ahler-Ricci soliton metric with soliton potential $f$.
\end{prop}




\section{Cohomogeneity one Ricci solitons under the Heisenberg group}\lb{Heisen-sols}

\subsection{The equations}
We now consider the special case of gradient K\"ahler-Ricci solitons for the action of the Heisenberg group. Note that in this case it will turn out that the $m-1$ non-trivial shear equations are satisfied
by a zero shear operator. And indeed we will be able to solve for the metric explicitly, in accordance
with the state of affairs described in the introduction.

For the Heisenberg group, one can take
\[
\mathlarger{\Gamma}_{iz}^i=\mathlarger{\Gamma}_{zi}^i=0,\qquad \mathlarger{\Gamma}_{ii}^z=1,\quad i=1,\ldots n,
\]
so that \Ref{Gam-ab} holds automatically.
The constants $\ell_i$ of \Ref{a-eq-b} can now be chosen freely, and we take them all to equal $1$. Also, from \Ref{uniform} we see that
for all $i=1,\ldots n$, $a_i^2=a_{\1}^2$. Thus $\al=a_{\1}^{2n}c$ and equations \Ref{1'''}-\Ref{K3'}
take the form
\begin{align}
2\fr{a_{\1}'}{a_{\1}}&=a_{\1}^{2n-2} c^2,\lb{a1}\\
2\fr{\,c'}{\textstyle c}&=-na_{\1}^{2n-2}c^2-2\lam a_{\1}^{2n+2}+ka_{\1}^{2n}c^2.\lb{c}
\end{align}
The metric has the form
\[
g=a_{\1}^2\sum_{i=1}^n(\sig_i^2+\rho_i^2)+c^2\zeta^2+a_{\1}^{4n}c^2\,dt^2.
\]
As in \cite{a-m2}, we now make the change of variables $a_{\1}^2\,dt=dq$. Then,
setting $\phi(q):=a_{\1}^2$, we see from \Ref{a1} that, with the prime denoting from now on
the derivative with respect to $q$,
\[
\phi'(q)=2a_{\1}\fr {da_{\1}}{dq}=2a_{\1}\fr {da_{\1}}{dt}\fr{dt}{dq}=2a_{\1}\fr {da_{\1}}{dt}\fr{1}{a_{\1}^2}=a_{\1}^{2n-2}c^2,
\]
and hence
\[ \phi'(q)\,dq^2=a_{\1}^{2n+2}c^2\,dt^2. \]
It follows that the metric takes the form
\be\lb{g-phi}
g=\sum_{i=1}^n\phi(q)(\sig_i^2+\rho_i^2)+\fr{\phi'(q)}{\phi(q)^{n-1}}\zeta^2+\phi(q)^{n-1}\phi'(q)\,dq^2.
\end{equation}
with K\"ahler form
\[
\om=d(\phi(q)\zeta).
\]
Note that equation~\Ref{a1} becomes an identity with these choices, while equation \Ref{c} yields
\begin{align*}
\Big(\fr{\phi'(q)}{\phi(q)^{n-1}}\Big)'\Big/\Big(\fr{\phi'(q)}{\phi(q)^{n-1}}\Big)&
=2c\fr{dc}{dt}\fr{dt}{dq}\fr 1{c^2}=\left(-na_{\1}^{2n-2}c^2-2\lam a_{\1}^{2n+2}+ka_{\1}^{2n}c^2\right)\fr1{a_{\1}^2}\\
&=-n\fr{\phi'(q)}{\phi(q)}-2\lam\phi(q)^n+k\phi'(q),
\end{align*}
or equivalently, 
\be\lb{sol-Heis}
\fr{(\phi^2)''}{(\phi^2)'}=-2\lam\phi^n+k\phi'.
\end{equation}
Also, from \Ref{f-pr}
\be\lb{f-phi}
\fr{df}{dt}=f'\fr{dq}{dt}=f'\phi=k\phi\phi'.
\end{equation}

In summary
\begin{prop}
Let $(M^{2n+2},g)$ be a cohomogeneity one manifold under a proper action of
the Heisenberg group $\mathcal{H}_{2n+1}$ with $g$ of the
form \Ref{g-phi}. If $\phi$ satisfies \Ref{sol-Heis} then $g$ is a gradient
K\"ahler-Ricci soliton metric with soliton potential $f$ given as an affine function
of $\phi$, so long as $f$ is either constant, or $\phi$, when considered as
a function of $\ta$ such that $q'(\ta)=(2\phi(q)^{n-1}\phi'(q))^{-1/2}$,
has only degenerate critical points.
\end{prop}
The last clause simply involves computing the change of variable from $\ta$ to $q$.

\subsection{A completeness theorem}\lb{comp}
Our goal is to find complete expanding gradient K\"ahler-Ricci soliton metrics
in the case where $\mathcal{G}=\mathcal{H}_{2n+1}$ is the $2n+1$-dimensional
Heisenberg group.
In this simple case of the construction given in the beginning of Section~\ref{sec:coho},
the cohomogeneity one manifold $M$ will only have regular orbits, and at any point $p$ the isotropy is
trivial, giving $M=\mathcal{H}_{2n+1}\times I$ for some open interval $I$, with the action given
by left multiplication on the first factor.

To set the stage, consider the metric
$g$ given by \Ref{g-phi} for $\phi$ satisfying~\Ref{sol-Heis}
with $\lam=-1$ and $k=-1$, i.e.
\[
\fr{(\phi^2)''}{(\phi^2)'}=2\phi^n-\phi'.
\]
This has a first integral
\be\lb{phi-1ord1}
\phi'=2(-1)^{n+1}(n+1)!\fr1{\phi}\Big[{\textstyle\sum_{k=0}^{n+1}\fr{(-1)^k}{k!}}\phi^k-e^{-\phi}\Big]
=:F_n(\phi),
\end{equation}
where the coefficient of the exponential is a particular
choice of an integration constant. The right hand side of this equation will be denoted $F_n(\phi)$.
Now, positive definiteness of $g$ holds if both
$\phi>0$ and $\phi'>0$. But $H_n(\phi)=\sum_{k=0}^{n+1}\fr{(-1)^k}{k!}\phi^k-e^{-\phi}$
is zero at $\phi=0$, and has a positive/negative derivative for $\phi>0$ if $n$ is odd/even,
as can be shown by induction starting at $n=-1$. Thus $\phi'>0$ if $\phi>0$ and therefore
$g$ is positive definite whenever $\phi>0$.

Fix $\phi_0>0$. Since $F_n(\phi)=2(-1)^{n+1}(n+1)!H_n(\phi)/\phi$ is positive and continuous on $(0,\infty)$, there exists a unique solution $\phi(q)$ to \Ref{phi-1ord1} with initial condition $\phi(q_0)=\phi_0\in(0,\infty)$ defined on the interval
\be\lb{domain0}
(q_a,q_b)=\left(\int_{\phi_0}^{0} 1/F_n(\phi)\,d\phi,\int_{\phi_0}^\infty 1/F_n(\phi)\,d\phi\right).
\end{equation}
Note that on $(q_a,q_b)$, both $\phi$ and $\phi'$ are positive, so the metric
$g$ is well-defined for $q$ in this interval. Note also that employing
L'H{\^o}pital's rule to $F_n(\phi)$ shows that $\phi'(q)$, like $\phi(q)$, approaches
zero as $q\to q_a$. Additionally, we have the following elementary lemma, whose statement and
method of proof will be alluded to a number of times later on.
\begin{lemma}\lb{taylor}
$\phi'(q)$, like $\phi(q)$, is increasing in $(q_a,q_b)$.
\end{lemma}
\begin{proof}
We consider the easily verifiable formula
\be\lb{F-der}
F_n'(\phi)=(n+1)F_{n-1}(\phi)-F_n(\phi)/\phi
\end{equation}
(where the prime is the $\phi$-derivative). As $\phi''(q)=F_n'(\phi(q))F_n(\phi(q))$,
to prove the latter claim it is enough to show $F_n'(\phi)>0$ for $\phi>0$. Unpacking this using
\Ref{F-der} and the definition of $F_n$ reveals that in case $n$ is odd, it suffices to show that
\be\lb{tay-pol}
P_{n+1}(\phi)+\phi P_n(\phi)<(\phi+1)e^{-\phi},
\end{equation}
where $P_n$ is the $n$-th Taylor polynomial
of $e^{-\phi}$. But the left hand side is just the $(n+1)$-th Taylor polynomial of the right hand
side.
The two sides have $n+1$ equal derivatives at $\phi=0$, whereas for $\phi>0$ we have
\begin{align*} [P_{n+1}(\phi)+\phi P_n(\phi)]^{(n+1)}&=(-1)^{n+1}+(-1)^{n}(n+1)=-n\\
&<(-1)^{n+1}(\phi-n)e^{-\phi}=[(\phi+1)e^{-\phi}]^{(n+1)}.
\end{align*}
Working backwards inductively, each $k$-th derivative, $k=n,n-1,\ldots,1,0$ of the left hand side
of \Ref{tay-pol} is smaller than the $k$-th derivative of the right hand side for $\phi>0$. Thus
\Ref{tay-pol} holds as required.
The proof for $n$ even is similar.
\end{proof}

In the proof of the next theorem we will make use of the fact that, as $\phi$ is monotone
on $(q_a,q_b)$, its inverse $q(\phi)$ is well defined on $(0,\infty)$.
\begin{thm}\lb{sol-Heis2}
Let $\phi(q)$ be a solution to \Ref{phi-1ord1}
with initial condition $\phi(q_0)=\phi_0$ with $\phi_0>0$,
defined on the interval \Ref{domain0}.
Then there exists a complete expanding gradient K\"ahler-Ricci soliton $(g,f)$
on $M=\mathcal{H}_{2n+1}\times (q_a,q_b)$,
such that $g$ is given by \Ref{g-phi}, while up to an additive constant,
$f=-\phi$.
\end{thm}
The proof will be given in the next subsection. A fully explicit
form for the metric appears later, see \Ref{explicit}.

Regarding the method of proof of completeness of metrics admitting a non-compact cohomogeneity
one group action by isometries, we note the following. The degree of closeness of the group to being abelian can be measured by the number of vanishing Lie bracket relations in some basis for its Lie algebra. Forming the analog of this basis as part of a suitable left-invariant orthogonal frame on the manifold, this number corresponds to the number of closed, hence locally exact, $1$-forms in the dual coframe. The primitives of these forms can be used as coordinate functions forming part of a coordinate system on an open and dense chart domain. We can call these the ``trivial" coordinates as their values along a finite length curve are fairly easy to bound. This approach was already employed in \cite{a-m2,mr}. However, in the first of these references we were only able to prove completeness for a K\"ahler-Einstein metric under the action of a {\em quotient} of the dimension $3$ Heisenberg group, with the quotient being employed in part to mod-out the remaining ``nontrivial" coordinate function. But as we will see, for K\"ahler-Ricci solitons under a Heisenberg group action, we bound that coordinate function as well, in spite of the nontriviality of the coordinate expression of the associated coframe $1$-form. One however can avoid the need to do so and still prove completeness, see
Remark~\ref{quotient}.

\vspace{-.205in}
\subsection{Length of an escaping curve}
As the local part has already been established, aside from the properties of $f$ and its gradient,
all that remains is to prove the completeness claim for $g$.
For this we have
\begin{prop}\lb{escape1}
A finite length curve on $(M,g)$ of Theorem~\ref{sol-Heis2}
cannot leave every compact set.
\end{prop}
\begin{proof}
It follows from \Ref{domain0} that
$\lim_{q\to q_b^-}\phi(q)=\infty$.
Moreover, for the inverse $q(\phi)$ of $\phi$ on $(0,\infty)$, we have
\be\lb{far1}
\int_{q_0}^{q_b}\sqrt{\phi^{n-1}(q)\phi'(q)}\,dq=\int_{\phi_0}^{\infty}\sqrt{\phi^{n-1}q'(\phi)}\,d\phi
=\int_{\phi_0}^{\infty}\sqrt{\phi^{n-1}/F_n(\phi)}\,d\phi=\infty
\end{equation}
as the Taylor remainder estimate for $H_n(\phi)$ easily shows $\sqrt{\phi^{n-1}/F(\phi)}$ is larger than a constant multiple of $1/\phi$. For the same reason
\be\lb{other-far}
\int_{q_a}^{q_0}\sqrt{\phi^{n-1}(q)\phi'(q)}\,dq=\int_0^{\phi_0}\sqrt{\phi^{n-1}q'(\phi)}\,d\phi
=\int_0^{\phi_0}\sqrt{\phi^{n-1}/F_n(\phi)}\,d\phi=\infty.
\end{equation}

We now choose a left-invariant frame for $\mathcal{H}_{2n+1}$ which is given
in coordinates $x_i$, $y_i$, $z$ by $X_i=\partial_{x_i}$, $Y_i=\partial_{y_i}+x_i\partial_z$,
$Z=-\partial_z$, to which we will add on $M$ the vector field $\partial_q$.
The domain of these coordinate functions is open and dense in $M$.
The corresponding left invariant coframe
is  $\sig_i=dx_i$, $\rho_i=dy_i$, $\zeta=\sum_i x_idy_i-dz$, $dq$.
Given a smooth curve $\gamma(s):I\to M$, with $I=[p,r]$ a closed interval, having
coordinate presentation $(x(s),y(s),z(s),q(s))$, we have
\begin{align*}
\gamma'&=\sum_{i=1}^n(x_i'\partial_{x_i}+y_i'\partial_{y_i})+z'\partial_z+q'\partial_q\\
&=\sum_{i=1}^n(\sig_i(\gamma')X_i+\rho_i(\gamma')Y_i)+\zeta(\gamma')Z+dq(\gamma')\partial_q
\end{align*}
so that for each $s\in I$ and each $i=1,\ldots n$
\begin{align*}
g\Big(\gamma'(s),\fr{X_i}{|X_i|}\Big|_{\gamma(s)}\Big)&=x_i'(s)\sqrt{\phi(q(s))},\\
g\Big(\gamma'(s),\fr{Y_i}{|Y_i|}\Big|_{\gamma(s)}\Big)&=y_i'(s)\sqrt{\phi(q(s))},\\
g\Big(\gamma'(s),\fr{\partial_q}{|\partial_q|}\Big|_{\gamma(s)}
\Big)&=q'(s)\sqrt{\phi(q(s))^{n-1}\phi'(q(s))}.
\end{align*}
It follows via the Cauchy-Schwartz inequality that for the length $L(\gamma)$ of $\gamma$
one has lower bounds
\begin{align}
L(\gamma)=\int_p^r|\gamma'(s)|\,ds&\ge\inf_I\Big(\sqrt{\phi(q(s))}\Big)\,\Big|\!\int_p^r x_i'(s)\,ds\Big|,\lb{x-prime1}\\
L(\gamma)&\ge\inf_I\Big(\sqrt{\phi(q(s))}\Big)\, \Big|\!\int_p^r y_i'(s)\,ds\Big|,\lb{y-prime1}\\
L(\gamma)&\ge\Big|\int_p^r\sqrt{\phi(q(s))^{n-1}\phi'(q(s))}q'(s)\,ds\Big|,\lb{q-prime1}
\end{align}
for $i=1,\ldots n$.
Now if $\gamma$ has finite length, \Ref{q-prime1} via the change of variable $q=q(s)$, along with \Ref{far1} and \Ref{other-far}, imply that $q(s)$ is bounded away from $q_a$ and $q_b$. Thus $\phi(q)$ is bounded away from zero along the curve. Therefore, \Ref{x-prime1}-\Ref{y-prime1} imply that $x_i(s)$ and $y_i(s)$ are bounded, for any $i=1,\ldots n$.

For the coordinate $z$, we have
\begin{align*}
g\Big(\gamma'(s),\fr Z{|Z|}\Big|_{\gamma(s)}\Big)&=\sqrt{F_n(\phi(q(s)))/\phi^{n-1}(q(s))}\zeta(\gamma'(s))\\
&=\sqrt{F_n(\phi(q(s)))/\phi(q(s))^{n-1}}\Big(\sum_{i=1}^n x_i(s)y_i'(s)-z'(s)\Big).
\end{align*}
So
\[
L(\gamma)\ge\inf_I\sqrt{F_n(\phi(q(s)))/\phi(q(s))^{n-1}}\int_p^r\Big|\sum_{i=1}^n (x_i(s)y_i'(s))-z'(s)\Big|\,ds.
\]
First, the infimum of the square root term in this expression
is bounded away from zero as before: we already know that
$q(s)$ is bounded away from $q_a$ and hence
$\phi(q(s))$ is bounded away from zero, whereas
$\sqrt{F_n(\phi)/\phi^{n-1}}$ is positive for $\phi>0$,
and is increasing as a function of $\phi$
due to an argument similar to the proof of Lemma~\ref{taylor}.

Next, if $|x_i(s)|\le M_i$
then
\[
|\sum_{i=1}^n (x_i(s)y_i'(s))-z'(s)|\ge |z'(s)|-|\sum_{i=1}^n x_i(s)y_i'(s)|\ge
|z'(s)|-\sum_{i=1}^n M_i|y_i'(s)|
\]
Now noting that in fact $\int_p^r |y_i'(s)|\,ds\le L(\gamma)/\inf\limits_I\Big(\sqrt{\phi(q(s))}\Big)$ for all $i=1,\ldots n$,
we see that
\begin{multline*}
\Big|\int_p^r z'(s)\,ds\Big|\le\int_p^r|z'(s)|\,ds\\
\le L(\gamma)/\inf_I\sqrt{F_n(\phi(q(s)))/\phi(q(s))^{n-1}}
+\left(L(\gamma)/\inf_I\Big(\sqrt{\phi(q(s))}\Big)\right)\sum_{i=1}^nM_i.
\end{multline*}
Thus $z(s)$ is bounded as well.

It follows that any curve with domain a closed interval of length $\le L$ starting
at a given point has domain contained in a (compact) four dimensional
coordinate cube. Hence the same holds if the curve is instead defined
on a semi-closed  bounded or unbounded interval.

\end{proof}

By  Proposition~\ref{escape1}, a curve 
that leaves every compact set has length bounded below by any finite number, i.e. its length is
infinite. Hence the metric $g$ in Theorem~\ref{sol-Heis2} is complete.

As for the remaining characterization of $f$, we know from
\Ref{f-phi} that $f'(q)=
-\phi'(q)$, so that up to an additive constant $f$ coincides with $\phi$.
Therefore $f'(\ta)=-\phi'(\ta)=-\phi'(q)\fr{dq}{dt}\fr{dt}{d\ta}=-\sqrt{\phi'(q)/(2\phi^{n-1}(q))}
=-\sqrt{F_n(\phi)/(2\phi^{n-1})}<0$, so $f(\ta)$ has no critical points.
This completes the proof of Theorem~\ref{sol-Heis2}.

\begin{remark}\lb{quotient}
Note that a simpler version of the proof of Theorem~\ref{sol-Heis2},
bounding only the ``trivial" coordinates $x_i$, $y_i$, will also apply to the case in which the
group acting is the quotient group considered in \cite{a-m2}, namely
$\mathcal{H}_{2n+1}/\tilde{\mathbb{Z}}$, where $\tilde{\mathbb{Z}}$ is the discrete subgroup of the center of $\mathcal{H}_{2n+1}$, given by
\[
\tilde{\mathbb{Z}}:= \left\{\begin{bmatrix} 1 & \bf{0} & 2\pi k\\ 0 & \,\mathrm{I}_n & \bf{0}\\ 0 & 0 & 1  \end{bmatrix} |\ k\in\mathbb{Z} \right\}.
\]
For this quotient group the center $K$ is isomorphic to $SO(2)$,
whose transitive action on the unit circle $S^{1}$ extends to a linear action on $V\!\smallsetminus\!\{0\}:=\mathbb{R}^{2}\!\smallsetminus\!\{0\}$.
Employing the notation of Section~\ref{sec:coho}, the manifold in that case
has the form $(\mathcal{H}_{2n+1}/\tilde{\mathbb{Z}})\times_K (D^2\!\smallsetminus\!\{0\})$,
where $D^2\subset V$ is the $2$-disk.
However, $\mathcal{H}_{2n+1}\times I$, $I$ an open interval, is the universal
covering for this manifold, so its completeness would then be implied as well
from such a proof. This alternative route to a proof would save the need
to explicitly bound the $z$ coordinate as we have done.

A different quotient of $M$ will be considered in subsection~\ref{asymp}.
\end{remark}

It is known that for a complete gradient soliton metric, the soliton vector field
is also complete \cite{zh}. We give an independent proof of this for completeness, and because we will later allude to one of the ingredients of the proof.
\begin{prop}
For $f$ as in Theorem~\ref{sol-Heis2}, $\n f$ is complete.
\end{prop}
\begin{proof}
We have
\begin{align*}
\n f=f'(\ta)(\kk-\tT)&=\fr1{\sqrt{2}\al}f'(t)(\kk-\tT)=-\fr1{\sqrt{2}\al}\al c\fr{\sqrt{2}}{\al}\partial_t\\
&=-\fr c{\al}\fr{dq}{dt}\partial_q=-\fr{a_1^2c}{\al}\partial_q=-\fr1{a_1^{2n-2}}\partial_q
=\fr1{\phi(q)^{n-1}}\partial_q.
\end{align*}
For a curve $\gamma$ as before, the only limitation on the range of the
parameter $s$ may come from the equation $q'(s)=1/\phi(q(s))^{n-1}$.
But
\[
\tilde{s}-s_0=\int_{s_0}^{\tilde{s}}\phi(q(s))^{n-1}q'(s)\,ds
=\int_{q_0}^{\tilde{q}}\phi(q)^{n-1}\,dq=\int_{\phi_0}^{\tilde{\phi}}\fr{\phi^{n-1}}{F_n(\phi)}\,d\phi,
\]
with $\tilde{\phi}=\phi(\tilde{q})=\phi(q(\tilde{s}))$. For $\tilde{\phi}=\infty$ or $\tilde{\phi}=0$,
the integral on the right diverges, respectively, to $\pm\infty$, as can be seen by showing
that
\be\lb{better}
\phi^{n-1}/F_n(\phi)>1/(2\phi).
\end{equation}
This follows due to the formula
\be\lb{induc}
F_n(\phi)=-(n+1)F_{n-1}(\phi)+2\phi^n
\end{equation}
which is immediately verifiable from the definition of $F_n(\phi)$.
From this we see that
\[
\frac{\phi^{n-1}}{F_n(\phi)}=\frac{\phi^{n-1}}{-(n+1)F_{n-1}(\phi)+2\phi^n}>
\frac{\phi^{n-1}}{2\phi^n}=\frac1{2\phi}
\]
because both $F_{n-1}$ and $F_n$ are positive.
\end{proof}

\subsection{Curvature and other properties}\lb{props}

\subsubsection{Explicit metric form}
Observe that $\phi(q)$, being obviously smooth, is increasing in $(q_a,q_b)$, so that
one can use $\phi$ as a coordinate, giving an explicit form for these Ricci solitons:
\be\lb{explicit}
\begin{aligned}
&g=\sum_{i=1}^n\phi(\sig_i^2+\rho_i^2)+\fr{F_n(\phi)}{\phi^{n-1}}\zeta^2
+\fr{\phi^{n-1}}{F_n(\phi)}\,d\phi^2,\\[2pt]
&F_n(\phi)=2(-1)^{n+1}(n+1)!\fr1{\phi}
\Big[{\textstyle\sum_{k=0}^{n+1}\fr{(-1)^k}{k!}}\phi^k-e^{-\phi}\Big],\\[2pt]
&f=-\phi+\mathrm{constant},
\end{aligned}
\end{equation}
where $\phi\in\mathbb{R}^+$.

\subsubsection{Ricci curvature}
Combining \Ref{Ric-frame}, \Ref{1}, \Ref{2} and \Ref{JdfK1}, we see that the only nonzero
components of the Ricci form satisfy
\[
\rho(\xx_i,\yy_i)=\lam +Nf'(\ta),\qquad
\rho(\kk,\tT)=\lam-2f''(\ta).
\]
One computes
\begin{align*}
\rho&=\fr12\left(\fr{F_n(\phi)}{\phi^n}-2\right)\sum_{i=1}^n\hat{\xx}_i\we\hat{\yy}_i
+\fr12\left(\fr{F_n'(\phi)}{\phi^{n-1}}-(n-1)\fr{F_n(\phi)}{\phi^n}-2\right)\hat\kk\we\hat\tT\\
&=\fr12\left(\fr{F_n(\phi)}{\phi^n}-2\right)\sum_{i=1}^n\hat{\xx}_i\we\hat{\yy}_i
+\fr12\left(-(\phi+n)\fr{F_n(\phi)}{\phi^n}+2\phi-2\right)\hat\kk\we\hat\tT\\
&=\fr12\left(\fr{F_n(\phi)}{\phi^n}-2\right)\phi\sum_{i=1}^n\sig_i\we\rho_i
+\fr12\left(-(\phi+n)\fr{F_n(\phi)}{\phi^n}+2\phi-2\right)\,d\phi\we\zeta,\\
\end{align*}
where the prime is the $\phi$-derivative, and the second step follows from the easily
verifiable formula
\[ F_n'(\phi)=-\fr{\phi+1}{\phi}F_n(\phi)+2\phi^n. \]
Since we have seen that $F_n(\phi)/\phi^{n-1}<2\phi$, it easily follows that
the Ricci curvature is negative definite. On the other hand the components of
the Hessian of $f$ are $-Nf'=-F_n(\phi)^{3/2}/(\sqrt{2}\phi^{(n+1)/2})$
and $2f''=-2F_n'(\phi)$, so the Hessian is also negative definite. It follows that
\[ -g<\mathrm{Ric}<0. \]

\subsubsection{Scalar Curvature}
The scalar curvature is thus
\begin{align*}
\mathrm{Scal}&=\fr{F_n(\phi)}{\phi^n}+\fr{F_n'(\phi)}{\phi^{n-1}}-2n-2\\
&=-\fr{F_n(\phi)}{\phi^{n-1}}+2(\phi-n-1),
\end{align*}
and satisfies
\[ -2m<\mathrm{Scal}<0, \]
with the lower bound also being the $\phi\to 0^+$
limit of $\mathrm{Scal}$. This lower bound is known to be sharp for normalized nontrivial
expanding gradient Ricci solitons \cite{prs,zhs}.

\subsubsection{Bound on the distance function}
Our bound \Ref{q-prime1} holds for any curve, hence in particular for a minimizing geodesic, so
that together with~\Ref{better} it gives a lower bound on the distance function:
\begin{align}
d_g(o,p)\ge \Big|\int_{q_o}^{q_p}\sqrt{\phi^{n-1}(q)\phi'(q)}\,dq \Big|&=
\Big|\int_{\phi_o}^{\phi_p} \sqrt{\phi^{n-1}/F_n(\phi)}\,d\phi\Big|\nonumber\\
&>\Big|\int_{\phi_o}^{\phi_p}\sqrt{1/2\phi}\,d\phi\Big| = \Big|\sqrt{2}\phi^{1/2}\Big|_{\phi_o}^{\phi_p}\Big|.\lb{dist}
\end{align}
Thus $-f=\phi$ grows asymptotically at most as a quadratic expression in
the distance function. However there is no such quadratic {\em lower bound}, as
the Ricci curvature does not decay quadratically near $\phi=0$,
a conclusion one can also deduce from the classification results in \cite{cds}.
Note that $f$ is not a proper function. See the next subsection for more on the
asymptotics of the solitons.

\subsubsection{The asymptotic scalar curvature ratio}
It follows from considerations mentioned in the next subsection that $F_n(\phi)/\phi^{n-1}\to 0$
as $\phi\to0^+$. Therefore, for any base point $p$,
$\lim_{\phi\to0^+}2\mathrm{Scal}\Big(\phi^{1/2}-\phi_p^{1/2}\Big)^2=-4m\phi_p$.
By \Ref{dist}, the asymptotic scalar curvature ratio
$\stackrel[\scriptscriptstyle d_g(o,p)\to\infty]{}{\limsup}\mathrm{Scal}(p)d_g(o,p)^2$,
which is base-point independent, is bounded below by this number, which could be made
arbitrarily close to $0$ by a choice of base point. On the other hand this quantity
cannot be positive, as $\mathrm{Scal}<0$. It follows that the asymptotic scalar curvature ratio
vanishes.


\vspace{-.2in}
\subsubsection{Soliton potential behaviour}
Writing the metric in the form $dr^2+g_r$, with $r$ a primitive of $\sqrt{\phi^{n-1}/F_n(\phi)}$,
we see that $f'(r)=-\fr1{dr/d\phi}=-\sqrt{F_n(\phi)/\phi^{n-1}}<0$, and $f''(r)<0$ also holds, due to an argument mentioned earlier analogous to the proof of Lemma~\ref{taylor}. This is similar to
soliton potential estimates obtained in \cite{bdw, bdgw, w2} under the assumption that the
manifold has a singular orbit.

\subsection{Asymptotics}\lb{asymp}

Strictly speaking, the manifold $M$ in Theorem~\ref{sol-Heis2}
has one end without a standard type of asymptotics. However, one can
quotient $\mathcal{H}_{2n+1}$, by, say, the subgroup consisting
of matrices with integer coefficients, to obtain a compact
nilmanifold $N$. Then $N\times\mathbb{R}^+$ will have two ends,
and the quotient Ricci soliton metric will have different asymptotic behaviour
in each of them, which we now describe in the zeroth order, i.e. at the level
of the metric and not its covariant derivatives. The coordinates used below are those
employed in the proof of Theorem~\ref{sol-Heis2}.

Since $\fr{F_n(\phi)}{\phi^{n-1}}\sim 2\phi$ as $\phi\to\infty$,
as can be seen inductively via \Ref{induc},
the metric $g$ has at that end the asymptotics of a cone metric
over the quotient of a left invariant metric.
\begin{align*}
g\sim g_\infty&=\fr 1{2\phi}d\phi^2+2\phi\,\zeta^2+\phi\sum_{i=1}^n(\rho_i^2+\sig_i^2)\\
&=dr^2+r^2[\zeta^2+\tfrac12\sum_{i=1}^n(\rho_i^2+\sig_i^2)]\\
&=dr^2+r^2\Big[\Big(\sum_{i=1}^n x_i\,dy_i-dz\Big)^2+\tfrac12\sum_{i=1}^n(dx_i^2+dy_i^2)\Big],
\end{align*}
with $r=\sqrt{2\phi}$. At least for dimension four, $g_0$ has sectional curvature in the plane of $X_{\1}$, $Y_{\1}$ approaching zero as $r\to\infty$, whereas the sectional curvature in the plane of $\partial_r$, $\partial_z$ vanishes.

On the other end, $\fr{F_n(\phi)}{\phi^{n-1}}\sim \fr{2\phi^2}{n+2}$ as $\phi\to 0^+$,
as is easiest to see by directly applying L'h{\^o}pital's rule inductively,
so $g$ has the following asymptotics at that end
\begin{align*}
g\sim g_0&=\fr{n+2}{2\phi^2}d\phi^2+\fr{2\phi^2}{n+2}\,\zeta^2
+\phi\sum_{i=1}^n(\rho_i^2+\sig_i^2)\\
&=2^{-1}(n+2)\,dr^2+2(n+2)^{-1}e^{2r}\zeta^2+e^r\sum_{i=1}^n(\rho_i^2+\sig_i^2)\\
&=2^{-1}(n+2)\,dr^2+2(n+2)^{-1}e^{2r}\Big(\sum_{i=1}^n x_i\,dy_i-dz\Big)^2+e^r\sum_{i=1}^n(dx_i^2+dy_i^2),
\end{align*}
with $r=\log\phi$. This K\"ahler metric has constant holomorphic sectional curvature $-2/(n+2)$, 
as can be checked, for example, via methods similar to those we have been using for
the Ricci solitons (see next subsection for sectional curvature matters).
This number appears in the conjecture in the introduction, which will be discussed for $n=1$ at the end of the next subsection.

Finally, we note that the conclusions of this subsection regarding the two ends resemble that of the
(real) dimension $3$ Ricci soliton of \cite{r}, which also has one conical end along with one cuspoidal end.

\subsection{Sectional curvature bounds and singularity type}\lb{singul}

As the solitons given by \Ref{explicit} are expanding, the associated solution to
the Ricci flow, evolving by rescalings of pull-backs of the metric under diffeomorphisms,
is either of type IIb or of type III, in the sense of \cite{hm1}.
Which one of these it is, depends on the behavior of
the curvature norm along the flow solution. We outline a proof that it is of the
latter type, by way of giving sectional curvature bounds, including precise
ones in complex dimension two.

\subsubsection{Frame sectional curvatures}
The sectional curvatures of $g$ on the frame fields are given by
\begin{align}
&\mathrm{Sec}(\xx_i,\yy_i)=-2N^2=-F_n(\phi)/\phi^{n+1},\nonumber\\
&\mathrm{Sec}(\kk,\xx_i)=(N^2-2LN)/2=((n+1+\phi)F_n(\phi)/\phi^{n+1}-2)/4,\lb{sec}\\
&\mathrm{Sec}(\kk,\tT)=2(L'-L^2)=-[n(n+1)/2+n\phi+\phi^2/2]F_n(\phi)/\phi^{n+1}+n-1+\phi,\nonumber
\end{align}
where the prime denotes $\partial_\ta$, and the second of these expressions also equals $\mathrm{Sec}(\kk,\yy_i)=\mathrm{Sec}(\tT,\xx_i)=\mathrm{Sec}(\tT,\yy_i)$, and additionally is
a constant multiple of the curvature tensor when evaluated on the quadruple $\kk$, $\tT$, $\xx_i$, $\yy_i$. All other curvature values, unless determined by \Ref{sec} and the curvature
symmetries, in fact vanish.

\subsubsection{The curvature matrix}
These formulas are obtained as follows. We first derive the first expressions in each line of \Ref{sec}, given in terms of $N$ and $L$. Applying \Ref{brack1}-\Ref{rels2}, we see that the connection form is
\[\theta=\begin{pmatrix}
\theta_{00} & \theta_{01} & \theta_{02} & \cdots & \theta_{0n} \\
\theta_{10} & \theta_{11} & 0 & \cdots & 0 \\
\theta_{20} & 0 & \theta_{22} & \ddots & \vdots \\
\vdots & \vdots & \ddots & \ddots & 0 \\
\theta_{n0} & 0 & \cdots & 0 & \theta_{nn}
\end{pmatrix}\]
where the $\theta_{ij}$ are $2\times2$ blocks satisfying $\theta_{ij}=-\theta^T_{ji}$.
Since $B_i=C_i=F_i=G_i=0$ in the case of the Heisenberg group, and $a_i=b_i$ together with
\Ref{rels2} imply $A_i=D_i=-E_i=-H_i=N/2$, we get
\[\theta_{00}=\begin{pmatrix}
0 & L(\kf+\tf) \\[2pt]
-L(\kf+\tf) & 0
\end{pmatrix},\]
\[\theta_{ii}=\begin{pmatrix}
0 & \frac{N}{2}(\kf+\tf) \\[2pt]
\frac{-N}{2}(\kf+\tf) & 0
\end{pmatrix},\]
\[\theta_{0i}=\begin{pmatrix}
\frac{N}{2}(\xf_i-\yf_i) & \frac{N}{2}(\xf_i+\yf_i) \\[2pt]
-\frac{N}{2}(\xf_i+\yf_i) & \frac{N}{2}(\xf_i-\yf_i)   \\
\end{pmatrix}.\]
The curvature matrix is then given by
\[ \Omega = \theta\wedge\theta+d\theta.\]
Using \Ref{3} for the $\tau$ derivative of $N$ this gives
\[\Omega=\begin{pmatrix}
\Omega_{00} & \Omega_{01} & \Omega_{02} & \cdots & \Omega_{0n} \\
\Omega_{10} & \Omega_{11} & \theta_{10}\wedge\theta_{02} & \cdots & \theta_{10}\wedge\theta_{0n} \\
\Omega_{20} & \theta_{20}\wedge\theta_{01} & \Omega_{22} & \ddots & \vdots \\
\vdots & \vdots & \ddots & \ddots & \theta_{(n-1)0}\wedge\theta_{0n} \\
\Omega_{n0} & \theta_{n0}\wedge\theta_{01} & \cdots & \theta_{n0}\wedge\theta_{0(n-1)} & \Omega_{nn}
\end{pmatrix}\]
where
\begin{align*}
\Omega_{00}&=\sum_{i=0}^n\theta_{0i}\wedge\theta_{i0}+d\theta_{00},\\
&=(N^2-2LN)\begin{pmatrix}0&1\\ -1&0\end{pmatrix}\sum_{i=1}^n\xf_i\wedge\yf_i+(2L'-2L^2)\begin{pmatrix}0&1\\ -1&0\end{pmatrix}\kf\wedge\tf,\\
\Omega_{ii} &= \theta_{i0}\wedge\theta_{0i}+\theta_{ii}\wedge\theta_{ii}+d\theta_{ii},\\
&=-N^2\begin{pmatrix}0&1\\ -1&0\end{pmatrix}\sum_{j=1}^n(1+\delta_{ij})\xf_j\wedge\yf_j+(N^2-2LN)\begin{pmatrix}0&1\\ -1&0\end{pmatrix}\kf\wedge\tf,\\
\Omega_{0i} &= \theta_{00}\wedge\theta_{0i}+\theta_{0i}\wedge\theta_{ii}+d\theta_{0i},\\
&=\frac{N^2-2LN}{2}\left( \begin{pmatrix}1&0\\ 0&1\end{pmatrix}(\kf\wedge\xf_i+\tf\wedge\yf_i)+\begin{pmatrix}0&1\\ -1&0\end{pmatrix}(\kf\wedge\yf_i-\tf\wedge\xf_i)\right).
\end{align*}
Since the terms $\theta_{i0}\wedge\theta_{0j}$ for $i\ne j$ are proportional to $N^2$, this shows that the only terms appearing in the full curvature form are multiples
of the sectional curvatures given in \Ref{sec}, and from this the first equalities
of \Ref{sec} follows easily.

Next, the final formulas in each line of \Ref{sec} in terms of $F_n$ and $\phi$ are obtained
in a long calculation starting from the expressions for $L$ and $N$ in \Ref{L-N}
using $N_i=N$, $a_1=a_i=b_i$, $a_1^2=\phi$, $c^2=F_n(\phi)/\phi^{n-1}$, $\partial_\ta=2^{-1/2}a_1^{-2n}c^{-1}\partial_t$
and $\partial_t=\phi F_n(\phi)\partial_\phi$.

\subsubsection{Bounded sectional curvature and Singularity type}
We now claim that all the expressions appearing in \Ref{sec}
are bounded for $\phi\in(0,\infty)$. First,
$F_n(\phi)/\phi^{n+1}$ is positive, and bounded above by $2/(n+2)$ via the Taylor remainder
estimate for $e^{-\phi}$. Next, $F_n(\phi)/\phi^n$ is again positive, and is bounded above
by $2$, since we have seen that $\phi^{n-1}/F_n(\phi)>1/(2\phi)$. This shows that the
first two sectional curvatures above are bounded. For the third, it remains to show that
the expression $-(\phi^2/2)F_n(\phi)/\phi^{n+1}+\phi=\phi-F_n(\phi)/(2\phi^{n-1})$ is bounded. It
is bounded below by $\phi-\fr12(2\phi)=0$. Then the Taylor remainder estimate shows its absolute
value is no larger than the polynomial $\phi+\fr2{n+2}\phi^2$, making it bounded on any finite interval of the form $(0,s)$. Finally its exact expression shows it equals $\phi+(n+1)-\phi=n+1$
plus terms that are bounded above on any interval of the form $(r,\infty)$, $r>0$. This concludes
the proof of this claim.

It follows from this and from the lines after \Ref{sec} that not just the frame sectional curvatures, but {\em all} sectional curvatures of $g$ are bounded. Hence the corresponding Ricci flow solution has
sectional curvatures which are $O(t^{-1})$, where $t$ is the flow parameter, and thus it is of type III (cf. \cite{lo}).

Note that one can improve these bounds to show that the frame sectional curvatures
are negative, using the expressions in \Ref{sec}. For $\mathrm{Sec}(\xx_i,\yy_i)$
this follows from the above. For $\mathrm{Sec}(\kk, \xx_i)$ this can be shown by simplifying
the needed inequality using the definition of $F_n$, giving for odd $n$, for example,
$(n+1+\phi)P_{n+1}(\phi)-\phi^{n+2}/(n+1)!<(n+1+\phi)e^{-\phi}$,  where $P_{n+1}$ is the
$n+1$-th Taylor polynomial of $e^{-\phi}$. This inequality can then be
proved by differentiating both sides $n+2$ times and then using the exact same method
of the proof of Lemma~\ref{taylor}. Finally, for $\mathrm{Sec}(\kk,\tT)$, one isolates
$F_n/\phi^{n+1}$ from the required inequality, replaces $F_n$ by \Ref{induc},
then replaces $F_{n-1}(\phi)/\phi^n$ by the expression coming from the inequality
showing $\mathrm{Sec}(\kk,\xx_i)<0$, and proceeds from there to immediately verify the
claim.

The holomorphic bisectional curvature is similarly bounded, and
negative for each pair of the form $\kk+i\tT$, $\xx_i+i\yy_i$, $i=1,\ldots n$.

\subsubsection{Explicit sectional curvature bounds in dimension four}
In dimension 4 (the n=1 case) we now give, as mentioned in the introduction, sharp bounds on all
sectional curvatures.
\begin{prop}
For the Ricci soliton metric in \Ref{explicit} with $n=1$, the sectional curvature satisfies
\[ -2/3<\mathrm{Sec}(\al)<0. \]
\end{prop}
\begin{proof}
Let $\Lambda^2=\Lambda^2_+\oplus\Lambda^2_-$ be the usual decomposition of $2$-forms into self-dual
and anti-selfdual parts.
The curvature operator $\mathcal{R}$ acting on $\Lambda^2$ can be obtained from $\Omega$ using
\[ \Omega(e_i,e_j)_{kl} = \langle \mathcal{R}(e_i\wedge e_j), e_k\wedge e_l \rangle, \]
where we dropped here and below the hat notation for dual $1$-forms\footnote{Note that the matrix indices $kl$ can be related to the previously used indices of some $2\times 2$ matrix $\Omega_{pq}$ in
that $(k,l)\in\{(2p,2q), (2p+1,2q),(2p,2q+1), (2p+1,2q+1)\}$.}.
Writing it relative to the standard orthonormal basis of $\Lambda^2_\pm$, namely $\frac{1}{\sqrt{2}}(\kk\wedge\tT\pm\xx\wedge\yy),
\frac{1}{\sqrt{2}}(\kk\wedge\xx\mp\tT\wedge\yy),\frac{1}{\sqrt{2}}(\kk\wedge\yy\pm\tT\wedge\xx)$,
gives
\[ \begin{pmatrix} p & 0 & 0 & q & 0 & 0 \\
0&0&0&0&0&0\\
0&0&0&0&0&0\\
q&0&0&p-2r&0&0\\
0&0&0&0&r&0\\
0&0&0&0&0&r\end{pmatrix},\]
where
\begin{align*}p&=\frac{1}{2}(\mathrm{Sec}(\kk,\tT)+\mathrm{Sec}(\xx,\yy))+2\mathrm{Sec}(\kk,\xx) \\
q&=\frac{1}{2}(\mathrm{Sec}(\kk,\tT)-\mathrm{Sec}(\xx,\yy))\\
r&=2\mathrm{Sec}(\kk,\xx)
\end{align*}
This shows that $\mathcal{R}$ has two self-dual eigenforms with eigenvalue 0, and two anti-self-dual eigenforms
with eigenvalue $r$. On the orthogonal complement of these 4 eigenforms $\mathcal{R}$ is
\[ \begin{pmatrix} p & q \\
q&p-2r
\end{pmatrix},\]
The remaining two eigenvalues then have sum given by the trace $2p-2r=\mathrm{Sec}(\kk,\tT)+\mathrm{Sec}(\xx,\yy)<0$
and product given by the determinant $p^2-2pr-q^2=\mathrm{Sec}(\kk,\tT)\mathrm{Sec}(\xx,\yy)-4\mathrm{Sec}(\kk,\xx)^2
=\frac{4}{\phi^4e^\phi}(2\cosh(\phi)-2-\phi^2)>0$. Therefore the remaining two eigenvalues are negative.
Since a decomposable two-form $\alpha$ cannot be self-dual, it must be that $\mathrm{Sec}(\alpha)<0$.
Since it can easily be checked that any of the frame sectional curvatures approach zero as $\phi\to\infty$, zero is a sharp upper bound on the sectional curvature.

For a sharp lower bound on the sectional curvature, consider a unit-length decomposable two-form in the basis above, given by $\alpha=(a_1,a_2,a_3,b_1,b_2,b_3)$ where $\sum a_i^2=\sum b_i^2=\frac{1}{2}$, a condition that must hold for decomposable forms. Then
\[\mathrm{Sec}(\alpha)=a_1^2p+2a_1b_1q+b_1^2(p-2r)+b_2^2r+b_3^2r=a_1^2p+2a_1b_1q+b_1^2(p-3r)+\frac{r}{2}\]
is a quadratic function in the variables $a_1,b_1\in[-1/\sqrt{2},1/\sqrt{2}]$, which can also be written as
\[\mathrm{Sec}(\alpha)=\frac{(a_1+b_1)^2}{2}\mathrm{Sec}(\kk,\tT)+
\frac{(a_1-b_1)^2}{2}\mathrm{Sec}(\xx,\yy)+(1+2a_1^2-4b_1^2)\mathrm{Sec}(\kk,\xx).\]
It can be checked that the functions $\mathrm{Sec}(\kk,\tT)$, $\mathrm{Sec}(\xx,\yy)$, and $\mathrm{Sec}(\kk,\xx)$ are all negative {\em increasing} functions of $\phi\in\mathbb{R}^+$,
and satisfy
\[\lim_{\phi\to 0^+} \mathrm{Sec}(\kk,\tT) = -2/3, \quad \lim_{\phi\to 0^+} \mathrm{Sec}(\xx,\yy)=-2/3, \quad \lim_{\phi\to 0^+} \mathrm{Sec}(\kk,\xx) = -1/6. \]
The only coefficient which can be negative is $1+2a_1^2-4b_1^2$, so we split $[-1/\sqrt{2},1/\sqrt{2}]^2$ into two regions
\begin{align*} A&=\{(a_1,b_1)\in[-1/\sqrt{2},1/\sqrt{2}]^2\ |\ 1+2a_1^2-4b_1^2>0 \},\\
B&=\{(a_1,b_1)\in[-1/\sqrt{2},1/\sqrt{2}]^2\ |\ 1+2a_1^2-4b_1^2\le0 \}.\end{align*}
For $(a_1,b_1)\in A$, each term in $\mathrm{Sec}(\alpha)$ is increasing in $\phi$, so a lower bound can be found by taking $\phi\to0^+$ where
\[\lim_{\phi\to0^+}\mathrm{Sec}(\alpha)=-\frac{1}{6}-a_1^2.\]
This gives lower bound when $a_1^2=1/2$ so on $A\times\mathbb{R}^+$ we get the sharp lower bound $\mathrm{Sec}(\alpha)>-2/3$.

For $(a_1,b_1)\in B$,
\[ \mathrm{Sec}(\alpha)\ge f(a_1,b_1,\phi):=\frac{(a_1+b_1)^2}{2}\mathrm{Sec}(\kk,\tT)
+\frac{(a_1-b_1)^2}{2}\mathrm{Sec}(\xx,\yy).\]
Each term in $f$ is increasing in $\phi$, so a lower bound can be found by taking $\phi\to0^+$ where
\[\lim_{\phi\to0^+}f(a_1,b_1,\phi)=-\frac{2}{3}(a_1^2+b_1^2),\]
which has minimum $-2/3$ on $B$, when $a_1^2=b_1^2=1/2$.
Combining the results on $A$ and $B$, this shows that we have sharp bounds on the sectional curvature in dimension four,
\[ -\frac{2}{3} < \mathrm{Sec} < 0. \]
\end{proof}
Thus the conjecture in the introduction holds for $m=2$.


\section{Steady solitons}\lb{steady-sols}

We now develop the case of steady solitons ($\lambda=0$) in the $\ta$-dependent case.
Our setting is that of $(M,g,J)$ of Proposition~\ref{kah}, with $A_i,\ldots H_i, N_i, L$
and $f$ all functions of $\ta$. All equations up to and including \Ref{another} continue to hold.
For skew-solitons, equations \Ref{1}-\Ref{2} hold, in
which we set $\lambda=0$. Setting
\[
Q:=2L+{\textstyle\sum_{j=1}^n}(C_j-H_j+A_j-F_j),
\]
equation \Ref{1} yields two main cases to consider, which we will denote from now on (I)
and (II):
\be\lb{I-II1}
\begin{aligned}
\mathrm{(I)}&\ \ Q=f'(\ta),\\
\mathrm{(II)}&\ \ N_i=0,\quad i=1,\ldots n.
\end{aligned}
\end{equation}
These cases need not be mutually exclusive. Also,
we will not consider hypothetical mixed cases where (I) and all or parts of (II)
hold on different subsets.

For these two cases, equation \Ref{2} becomes
\be\lb{I-II2}
\begin{aligned}
\mathrm{(I)}&\ \ f''(\ta)-Lf'(\ta)=0,\\
\mathrm{(II)}&\ \ -L(Q+f'(\ta))+(Q+f'(\ta))'=0.
\end{aligned}
\end{equation}
To these we add the conditions for existence of a K\"ahler-Ricci soliton,
which we reproduce here:
\begin{align}
\label{sol1}f''(\tau)+Lf'(\tau)&=0,\\
\label{sol2}A_i+E_i&=0,\\
\label{sol3}D_i+H_i&=0,\\
\label{sol4}B_i+C_i+F_i+G_i&=0,\\
\lb{sol5}\text{$f(\ta)$ has only degenerate}&\text{ critical points,}
\end{align}
where the first four of these conditions only need to hold away from the
critical points of $f(\ta)$.

Note that case (I) in \Ref{I-II2}, together with \Ref{sol1} can be given
alternatively as
\be\lb{auxil}
f''(\ta)=0,\qquad Lf'(\ta)=0
\end{equation}
and the first of these implies
that away from the critical points of $f(\ta)$, and hence everywhere, $f$ is affine in $\ta$.
Thus if we are not in the trivial soliton case where $f$ is constant, $L$ must vanish.

Finally, one should remember that to this list of conditions and equations one must
add the K\"ahler conditions \Ref{brack1}-\Ref{rels2}.

\subsection{Cohomogeneity one steady solitons}

For cohomogeneity one solitons, consider first the K\"ahler conditions
\Ref{K3}-\Ref{another}. Away from the zeros of $f'(\ta)$.
\Ref{another} together with \Ref{L-N} and \Ref{sol4} implies as before \Ref{a-eq-b}, i.e.
\be\lb{a-eq-b1}
\text{$b_i=\ell_i a_i$ for a constant $\ell_i>0$ satisfying
$\mathlarger{\Gamma}_{iz}^i-\ell_i^2\mathlarger{\Gamma}_{zi}^i=0$, $i=1,\ldots n$.}
\end{equation}
Next we have in our two cases
\be\lb{constants}
\begin{aligned}
\mathrm{(I)}&\ \ \text{$c=c_0$ is constant},\\
\mathrm{(II)}&\ \ \text{$a_i$ are constant $i=1,\ldots n$},
\end{aligned}
\end{equation}
the first since $L=0$, while the second since $N_i=0$ implies
$\mathlarger{\Gamma}_{ii}^z=0$ in \Ref{K3}. Also, for case (II)
\Ref{f-pr} remains in effect:
\be\lb{f-pr1} \mathrm{(II)}\ \text{$f'=k\al c$ for a constant $k$,}\end{equation}
where the prime in this subsection denotes differentiation with respect to $t$.

Now the formula for $Q$ reads, due to \Ref{L-N}, \Ref{a-eq-b1} and \Ref{K3},
\[
Q=-\sqrt{2}\fr1\al\fr{c'}c+\fr1{\sqrt{2}c}\sum_{i=1}^n 2\ell_j^{-1}\mathlarger{\Gamma}_{jz}^j
-\fr1{\sqrt{2}}\sum_{j=1}^n\mathlarger{\Gamma}_{jj}^z\fr c{\ell_ja_j^2},
\]
where the first term on the right hand side is zero in case (I), while the last sum  is zero in case (II).
Therefore \Ref{I-II1}-(I) and \Ref{I-II2}-(II) become, in view of \Ref{auxil} and \Ref{constants},
and together with \Ref{K3} in case (I) and taking into account \Ref{f-pr1} in case (II), after minor simplification
\be\lb{both}
\begin{aligned}
\mathrm{(I)}&\ \ \fr1{c_0}\sum_{i=1}^n 2\ell_j^{-1}\mathlarger{\Gamma}_{jz}^j-c_0\sum_{j=1}^n\mathlarger{\Gamma}_{jj}^z\fr 1{\ell_ja_j^2}=K,\qquad
(a_i^2)'=\fr{c_0^2}{\ell_i}\mathlarger{\Gamma}_{ii}^z\prod_{k=1}^na_k^2,\\
\mathrm{(II)}&\ \ \fr{c'}c\Big(kc+\fr1{c}\sum_{i=1}^n 2\ell_j^{-1}\mathlarger{\Gamma}_{jz}^j-\fr{2}{\al}\fr{c'}c\Big)
+\Big(kc+\fr1{c}\sum_{i=1}^n 2\ell_j^{-1}\mathlarger{\Gamma}_{jz}^j-\fr{2}{\al}\fr{c'}c\Big)'=0,\\
\end{aligned}
\end{equation}
where $K$ is a constant and $i=1,\ldots n$.

In the first equation of \Ref{both} in case (I), the first sum is just a constant.
Moreover, recall that $\mathlarger{\Gamma}_{jz}^j=0$ for all $j\ne i$ if $\mathlarger{\Gamma}_{ii}^z\ne 0$. Thus that first sum contains at most one term.

In case (II), the equation in \Ref{both} simplifies to
\be\lb{II-simple}
\mathrm{(II)}\ \ k\beta c'+\fr{(c')^2}{c^3}-\fr{c''}{c^2}=0,
\end{equation}
where $\beta=\prod_{i=1}^na_i^2$ is constant by \Ref{constants}-(II).

\subsection{No solutions in case (I)}

Assume without loss of generality that $\mathlarger{\Gamma}_{ii}^z$ are nonzero for
$i=1,\ldots k$, where $1\le k\le n$, and any others are zero. Differentiating the first equation in case (I) of
\Ref{both} while substituting the last $n$ gives
\[
c_0^3\prod_{i=1}^na_i^2\sum_{j=1}^k(\mathlarger{\Gamma}_{jj}^z)^2\ell_j^{-2}a_j^{-4}=0
\]
which obviously cannot hold as $c_0$ and $a_i$, $i=1,\ldots n$ are nonzero for a metric. Thus case (I) yields no solutions,
unless it is a special case of case (II).

\subsection{Solutions in case (II)}

For given $k$, $\beta$, the general nonconstant positive solution to \Ref{II-simple}  is
\[
c = \Big(k_2 e^{k_1 t} + k\beta/k_1\Big)^{-1/2}
\]
where $k_1\ne 0$, $k_2$ are integration constants. The steady soliton metric has the form
\[
g=\fr{\beta^2\,dt^2+\zeta^2}{k_2 e^{k_1 t} + k\beta/k_1}
+\sum_{i=1}^n(1+\ell_i^2)a_i^2(\sig_i^2+\rho_i^2),
\]
where we recall that $a_i$, $i=1,\ldots n$ and $\beta$ are constants.
The length of a geodesic normal to the orbits is given by
\[
2\sqrt{\fr{\beta}{kk_1}}\tanh^{-1}\left.\left(\sqrt{\fr{k_1}{k\beta}}
\sqrt{k_2e^{k_1t}+\fr{k\beta}{k_1}}\right)\right|_{t_0}^t
\]
As the argument of the inverse hyperbolic tangent is nonnegative and monotone,
this function will approach a finite value at its infimum. Hence such a metric, when defined on a manifold with
no singular orbits, will be incomplete. By adopting the methods of \cite{vz} one can show the
same holds even with a singular orbit present.

\section{Complete cohomogeneity one K\"ahler-Ricci skew-solitons under $E(2)$}\lb{E2-skews}

In this section we show the existence of K\"ahler-Ricci skew-solitons in dimension four, on
a space that does not admit gradient K\"ahler-Ricci solitons of the kind we are considering in
this paper, namely the $\ta$-dependent case defined in subsection~\ref{ta-dep}.

For our space we take $(M,g)$ to be a cohomogeneity one $4$-manifold under the action
of the Euclidean group $E(2)$. For $g$ of the form \Ref{bianchiAmet} we have $n=1$, and denote $a:=a_{\1}$ and $b:=b_{\1}$.  We represent the Lie algebra of $E(2)$ via a frame with $\mathlarger{\Gamma}_{\1 \1}^z=\mathlarger{\Gamma}_{\1 z}^{\1}=1$,
and $\mathlarger{\Gamma}_{z\1}^{\1}=0$, so the condition $\mathlarger{\Gamma}_{\1 z}^{\1}-\ell_{\1}^2\mathlarger{\Gamma}_{z\1}^{\1}=0$ from \Ref{a-eq-b} is violated,
showing that indeed a $\ta$-dependent gradient K\"ahler-Ricci soliton does not exist in this case.
Finally, as we will be considering expanding skew-solitons, we take $\lambda=-1$.

To state our theorem, we make the following definition.\\
Let $\ep=\ep(b):[0,\infty)\to\mathbb{R}$ be a $C^\infty$ function
satisfying
\be\lb{ep}
\begin{aligned}
i)\ &\ep(0)=0,\qquad 0\le\ep(b),\\
ii)\ &\ep'(0)=0,\qquad
\ep'(b)>-2b,\\
iii)\ &\text{$\ep(b)$ has a smooth extension  to an even function on $\mathbb{R}$.}
\end{aligned}
\end{equation}

\begin{thm}\lb{complete}
Let $M$ be a $4$-manifold admitting a cohomogeneity one action of $E(2)$ with a single singular orbit and an $E(2)$-invariant metric of the form \Ref{bianchiAmet} with $n=1$.
Then for any function $\ep(b)$ as in \Ref{ep}, there exists on $M$ a complete expanding gradient
K\"ahler-Ricci skew-soliton which is cohomogeneity one under this action,
of the form $(g,f)$,
where
\be\lb{bianchiAmet1}
\begin{aligned}
g &= (abc)^2dt^2+a^2\sigma_1^2+b^2\sigma_2^2+c^2\sigma_3^2,\\
f'&=\fr{df}{dt}=2\ep(b)a^2,
\end{aligned}
\end{equation}
for certain functions $a$, $b$, $c$ of $t$.
\end{thm}
The proof follows very closely the proof of existence of K\"ahler-Einstein metrics
on such manifolds in \cite{mr}, but with more difficult estimates. The latter proof,
in turn, took its inspiration from \cite{d-s1}.
It will be given over a number of separate steps.

\subsection{The system}

We consider the system
\begin{align}
2a'/a&=-a^2+c^2,\lb{E2ODEa}\\
2b'/b&=a^2+c^2, \lb{E2ODEb}\\
2c'/c&=a^2-c^2+2(ab)^2+2\ep(b)a^2,\lb{E2ODEc}\\
f'&=2\ep(b)a^2, \lb{E2ODEd}
\end{align}
for $\ep$ as in \Ref{ep}. The first two equations are easily derived from the K\"ahler conditions
\Ref{K3}-\Ref{another}, while the third is just the specialization of \Ref{1''} to this case,
once we make the choice of $f'$ as in \Ref{E2ODEd}. Thus existence of a solution to this system
for positive $a$, $b$, $c$ will give the skew-solitons, and the main task is to find complete
solution metrics.

The system \Ref{E2ODEa}-\Ref{E2ODEc} decouples from the final equation \Ref{E2ODEd}.
Using \Ref{ep}i), ii), one of its equilibrium points is $(q,0,q)$, $q>0$,
and its linearization is the same as when $\ep\equiv 0$, with one negative, one zero and one positive
eigenvalue. By the Center Manifold Theorem, its unstable curve is smooth and has domain
$(-\infty,\eta)$. We will give a condition on $(a,b,c)$, all positive, which yields such a curve.
But first we eliminate other candidate conditions. To this end, we first record a few useful relations.
\begin{lemma}\lb{prelim}
For the system \Ref{E2ODEa}-\Ref{E2ODEc},
\begin{align*}
(ab)'&=abc^2, \\
(bc)'&=bca^2\left(1+b^2+\ep(b)\right), \\
(ac)'&=a^3c(b^2+\ep(b)), \\
\left(\frac{a}{b}\right)'&=-\frac{a^3}{b}, \\
(c^2-a^2)'&=-(c^2-a^2)(a^2+c^2)+2 a^2c^2(b^2+\ep(b)) \\
\end{align*}
\end{lemma}

\subsection{Solutions}
\begin{prop}\label{prop:solutions}
There are no complete metrics corresponding to solutions of \Ref{E2ODEa}-\Ref{E2ODEc} with maximal interval $(\xi,\eta)$, when $\xi$ is finite. Furthermore, the only solutions with maximal interval $(-\infty,\eta)$ are the unstable curves of the equilibrium points $(q,0,q)$, $q>0$, and these solutions satisfy $0\le c^2-a^2\le 2a^2(b^2+\ep(b))$.
\end{prop}
\begin{proof}
For an initial time $t_0$, we let $(\xi,\eta)$ be a maximal solution interval for the initial value problem for \Ref{E2ODEa}-\Ref{E2ODEc} with $a(t_0)=a_0$, $b(t_0)=b_0$, and $c(t_0)=c_0$.

Uniqueness of solutions to \Ref{E2ODEa}-\Ref{E2ODEc} implies that if any of
$a$, $b$, or $c$ are zero anywhere in $(\xi,\eta)$ then it is zero everywhere.
Accordingly we thus assume that $a$, $b$, and $c$ are all positive on $(\xi,\eta)$.
Then we see from Lemma~\ref{prelim} and \Ref{E2ODEb} that $ab$, $bc$, $ac$, and $b$ are all
increasing on $(\xi,\eta)$.

We consider the following cases:
\subsubsection*{Case 1: $c_0^2-a_0^2<0$} We first claim the following.\\[4pt]
Claim: In this case $a\to\infty$ as $t\to\xi^+$.\\[3pt]
\textit{Proof of claim:} Since
\[ (c^2-a^2)'=-(c^2-a^2)(c^2+a^2)+2a^2c^2(b^2+\ep(b)), \]
if $c^2-a^2<0$ then $(c^2-a^2)'>0$, thus $c^2-a^2<0$ for all $\xi<t<t_0$.
Therefore,
\begin{align*}
a' &= \frac{a}{2}(c^2-a^2), \\
a'' &= \frac{a}{4}[(c^2-a^2)^2-2(c^2-a^2)(c^2+a^2)+4a^2c^2(b^2+\ep(b))],
\end{align*}
showing that $a$ is decreasing and concave up on $(\xi,t_0)$.
Next, we always have $b'>0$, while
\[c'=\frac{c}{2}(a^2-c^2+2a^2(b^2+\ep(b)))>0, \]
shows that $c$ is increasing on $(\xi,t_0)$.
Therefore $b$ and $c$ are bounded on $(\xi,t_0)$.
Thus, as $(\xi,\eta)$ is the maximal solution interval,
$a$ could be bounded as $t\to\xi^+$ only if $\xi=-\infty$.
But since $a$ is concave up, $a\to\infty$ as $t\to\xi^+$ even when $\xi=-\infty$. \qed

Since $ab$ and $ac$ are increasing, they are bounded as $t\to\xi^+$ and $a\to\infty$, so $b\to0$, $c\to0$, and $ab\to k$ for some constant $k$.
Near $b=0$ we can expand $\ep(b)$ so that for some constants $l$ and $m_0$,
\[ |\ep(b) - lb^2| \le m_0b^4. \]
Multiplying by $a^2$ this becomes
\[ |a^2\ep(b) - la^2b^2| \le m_0a^2b^4. \]
This show that as $t\to\xi^+$, $a^2\ep(b)\to lk^2$.
Then as $t\to\xi^+$ the equations will take the asymptotic form
\begin{align*}
a'&=-\frac{1}{2}a^3 \\
b'&=\frac{1}{2}ba^2 \\
c'&=\frac{1}{2}c(a^2+2k^2(l+1)) \\
\end{align*}
the solution of which has asymptotic form
\begin{align*}
a&\simeq (t-\xi)^{-\frac{1}{2}},\\
b&\simeq b_0(t-\xi)^{\frac{1}{2}},\\
c&\simeq c_0(t-\xi)^{\frac{1}{2}},\\
\end{align*}
for some constants $b_0$ and $c_0$.
This shows that $\xi$ is finite in this case and
\[ \int_{\xi}^{t_0} abc\,dt <\infty,\]
hence the metric is not complete.

\subsubsection*{Case 2: $c_0^2-a_0^2>2a_0^2(b_0^2+\ep(b_0))$} We have, again, a similar claim.\\[4pt]
Claim: In this case $c\to\infty$ as $t\to\xi^+$.\\[3pt]
\textit{Proof of claim:} Since
\begin{align*} (c^2-a^2-2a^2(b^2+\ep(b)))'&=-(c^2-a^2)(c^2+a^2)-2a^2b^2c^2\\&+2a^4(\ep(b)-b\ep'(b)/2)-a^2c^2b\ep'(b), \\
&=-(c^2-a^2-2a^2(b^2+\ep(b)))(c^2+a^2)-4a^2b^2c^2\\ &-2a^4b^2-2a^2c^2\ep(b)-a^4b\ep'(b)-a^2c^2b\ep'(b),\\
&=-(c^2-a^2-2a^2(b^2+\ep(b)))(c^2+a^2)-2a^2b^2c^2\\ &-2a^2c^2\ep(b)-a^2b(\ep'(b)+2b)(a^2+c^2),
\end{align*}

if $c^2-a^2-2a^2(b^2+\ep(b))>0$ then $c^2-a^2-2a^2b^2>0$ so that via \Ref{ep}ii) $(c^2-a^2-2a^2(b^2+\ep(b)))'<0$, thus in this case, $c^2-a^2>2a^2(b^2+\ep(b))$ for all $\xi<t\le t_0$.
Therefore,
\begin{align*}
c' &= \frac{c}{2}(a^2-c^2+2a^2(b^2+\ep(b)), \\
c'' &= \frac{c}{4}[(a^2-c^2+2a^2(b^2+\ep(b)))^2+2(c^2-a^2-2a^2(b^2+\ep(b)))(c^2+a^2)\\&+4a^2b^2c^2+4a^2c^2\ep(b)+2a^2b(\ep'(b)+2b)(a^2+c^2)],
\end{align*}
showing via \Ref{ep}ii) that $c$ is decreasing and concave up on $(\xi,t_0)$.
Next, we always have $b'>0$, while
$a$ is increasing on $(\xi,t_0)$.
Therefore $a$ and $b$ are bounded on $(\xi,t_0)$.
Thus, as $(\xi,\eta)$ is the maximal solution interval,
this time $c$ could be bounded as $t\to\xi^+$ only if $\xi=-\infty$.
But since $c$ is concave up, $c\to\infty$ as $t\to\xi^+$
even when $\xi=-\infty$. \qed

Since $ac$ and $bc$ are increasing, they are bounded as $t\to\xi^+$ and $c\to\infty$, so $a\to0$, $b\to0$.
Then as $t\to\xi^+$ the equations will take the asymptotic form
\begin{align*}
a'&=\frac{1}{2}ac^2 \\
b'&=\frac{1}{2}bc^2 \\
c'&=-\frac{1}{2}c^3 \\
\end{align*}
which have the solution
\begin{align*}
a&\simeq a_0(t-\xi)^{\frac{1}{2}}\\
b&\simeq b_0(t-\xi)^{\frac{1}{2}}\\
c&\simeq (t-\xi)^{-\frac{1}{2}}\\
\end{align*}
for some constants $a_0$ and $b_0$.
This shows that $\xi$ is finite in this case and
\[ \int_{\xi}^{t_0} abc\,dt <\infty,\]
so the metric is not complete.

If $c^2-a^2<0$ or $c^2-a^2>2a^2b^2$ at any time, then a constant shift in $t$ will give one of the previous cases.
In both previous cases, $\xi$ is finite, but we know that the unstable curve of the equilibrium points $(q,0,q)$ must have $\xi=-\infty$.
The existence of such curves is guaranteed by the Center Manifold theorem.
Therefore we consider the final case:

\subsubsection*{Case 3: $0\le c^2-a^2 \le 2a^2(b^2+\ep(b))\textnormal{ for all }t\in(\xi,\eta)$} Here
$(\xi,\eta)$ is a maximal domain for the solution\\[4pt]
Claim: In this case $\xi=-\infty$. \\[3pt]
\textit{Proof of claim:}
If $t_0$ is in the maximum interval,
as $a$, $b$, and $c$ are all increasing, they are therefore all bounded on $(\xi,t_0)$.
Since $(\xi,\eta)$ is the maximal solution interval, we see that $\xi=-\infty$. \qed

As $a$, $b$, and $c$ are all increasing, it must be that they all approach finite non-negative limits as $t\to-\infty$.
Thus $(a,b,c)$ must approach an equilibrium point.
If $(a,b,c)\to(0,q,0)$ with $q>0$, then $a/b\to 0$ as $t\to-\infty$, but $a/b$ is decreasing and positive,
hence this cannot happen.

Therefore, when $t\to-\infty$ we see that $(a,b,c)\to(q,0,q)$ in this case, but we still need to rule out
the possibility that $q=0$. For this we compute the variation of $ac$ with respect to $b$:
\begin{align}
\frac{d(ac)}{db}&=\frac{2(\frac{a}{c})^2(ac)(b+\fr{\ep(b)}{b})}{(\frac{a}{c})^2+1}.\lb{ac-b}
\end{align}
Our assumption of an equilibrium point $(q,0,q)$ implies that $a/c\to 1$ when $b\to 0$.
Since in our case $c^2-a^2\ge 0$, we have
\[ \frac{a}{c}\le 1. \]
Employing this inequality in equation \Ref{ac-b} yields
\[
\frac{d(ac)}{db}\le (ac)\Big(b+\fr{\ep(b)}b\Big).
\]
By Gr\"onwall's inequality, if $ac\to 0$ when $b\to 0$ then $ac=0$ identically. As the latter
is not possible,
neither is $q=0$.

\end{proof}

\subsection{Distance computations}

In the following proposition in particular, the estimates diverge significantly from those
in~\cite{mr}.
\begin{prop}\label{prop:estimates}
Let $g$ be a Riemannian metric of the form \Ref{bianchiAmet1} on a manifold $M$, with $a$, $b$, $c$ a solution to \Ref{E2ODEa}-\Ref{E2ODEc} along an unstable curve of an equilibrium point $(q,0,q)$, $q>0$,
having maximal domain $I=(-\infty,\eta)$. Assume that the latter interval is also the range of the coordinate function $t$ on $M$.
Then for any point $p_0\in M$ with orbit through $p_0$ of principal type and $M^t$ a level set of $t$,
\[
\lim_{t\to -\infty}d_g(p_0,M^t)<\infty,\qquad
\lim_{t\to \eta}d_g(p_0,M^t)=\infty,
\]
where $d_g$ is the distance function induced by $g$.
\end{prop}
\begin{proof}
The union of the principal orbits forms an open dense set, $\tilde{M}$, so that
$\tilde{M}/\mathcal{G}$ is a smooth manifold of dimension 1. The function $t$ is
a smooth submersion from $\tilde{M}$ to $\tilde{M}/\mathcal{G}$. The metric
\[ (abc)^2dt^2\]
makes this into a Riemannian submersion.
The level sets of $t$ are orbits of $\mathcal{G}$ and for $t_0=t(p_0)$
\[d_g(p_0,M^{t_1}) = d_g(M^{t_0},M^{t_1}), \]
is the distance in the quotient manifold, where
\[ d_g(M^{t_0},M^{t_1}) = \left|\int_{t_0}^{t_1} abc\, dt\right|. \]

Asymptotically as $t\to-\infty$,
\begin{align*}
a &\simeq q\\
b &\simeq ke^{q^2t}\\
c &\simeq q
\end{align*}
This gives the asymptotic metric
\[ g\simeq k^2q^4e^{2q^2t}dt^2+q^2\sigma_1^2+k^2e^{2q^2t}\sigma_2^2+q^2\sigma_3^2, \]
and for $v=ke^{q^2t}$ this is just
\[ g\simeq (dv^2+v^2\sigma_2^2)+q^2(\sigma_1^2+\sigma_3^2). \]
In this coordinate, the endpoint $\xi=-\infty$ is at $v=0$, and we see that
\[ \lim_{t\to -\infty}d_g(p_0,M^t) = \int_{-\infty}^{t_0} abc dt = \int_0^{v_0} dv < \infty, \]
as claimed.

Now to understand the behavior at the $\eta$ side of the solution interval,
we examine the derivative of $a/c$ with respect to $b$
\be\lb{a/c-b}
\frac{d(\frac{a}{c})}{db}=
\frac{2(\frac{a}{c})}{b}\left(\frac{1-(\frac{a}{c})^2(1+b^2+\ep(b))}{(\frac{a}{c})^2+1}\right),
\end{equation}
Ignoring the case $a/c=0$, this equation has nullcline
\[\frac{a}{c}=\sqrt{\frac{1}{1+b^2+\ep(b)}},\]
Since \Ref{ep}ii) easily implies that the nullcline is always decreasing, and our solution starts at $\frac{a}{c}=1$ when $b=0$, we have, using also \Ref{ep}i)
\[ \frac{a}{c}\ge\sqrt{\frac{1}{1+b^2+\ep(b)}} \]
To find a better upper bound than $1$ for $a/c$, we consider the curve $\frac{a}{c}=\sqrt{\frac{k^2}{k^2+b^2}}$ and plug its expression into the slope field.
This gives slope
\[ \frac{2k^2+2b^2}{2k^2+b^2}(k^2-1)\frac{d}{db}\sqrt{\frac{k^2}{k^2+b^2}}
-2\Big(\fr{k^2}{k^2+b^2}\Big)^{3/2}\fr{\ep(b)}{b\Big(\fr{k^2}{k^2+b^2}+1\Big)}.\]
So for $k^2>2$, the slopes of the solutions along this curve are less, i.e. more negative,
than the slope of the curve. Since for our (non-equilibrium) solution, there is some $b_p$
in the range of $b$ for which one has $(\fr ac)(b_p)<1$, the graph of $a/c$
is below the graph of $\sqrt{\frac{k^2}{k^2+b^2}}$ for some $k=k_p>2$ at $b_p$, and hence for
all $b\ge b_p$. Therefore for $b\ge b_p$
\[ \sqrt{\frac{1}{1+b^2+\ep(b)}}\le \frac{a}{c}\le\sqrt{\frac{k_p^2}{k_p^2+b^2}} \]
Next we deduce an estimate for $ac$ in terms of $b$, for $b>b_p$. Using \Ref{ac-b} and the last inequalities, we have
\[ \frac{2(b+\fr{\ep(b)}b)}{2+b^2+\ep(b)} \le \frac{d(\ln(ac)}{db} \le 2\Big(b+\fr{\ep(b)}b\Big)\frac{k_p^2}{2k_p^2+b^2}, \]
and integrating this from $b=0$ to $b$, exponentiating and multiplying by $q^2$ gives
\[ q^2\exp\left(\int_0^b\frac{2\Big(\tilde{b}+
\fr{\ep(\tilde{b})}{\tilde{b}}\Big)}{2+\tilde{b}^2+\ep(\tilde{b})}\,d\tilde{b}\right)
 \le ac \le q^2\exp\left(\int_0^b 2\Big(\tilde{b}+\fr{\ep(\tilde{b})}{\tilde{b}}\Big)
 \frac{k_p^2}{2k_p^2+\tilde{b}^2}\,d\tilde{b}\right). \]
Note here that our assumption \Ref{ep}ii) does imply the integrands are integrable.
Now since the integrand on the left hand side is larger than $b/(2+b^2)$, we in fact
have
\[
ac\ge q^2(2+b^2)^{1/2}
\]
for $b>b_p$. We thus finally arrive at an estimate for $b'$ as
\[ b' = \frac{b}{2}(a^2+c^2) =\frac{b}{2}(ac)\left(\frac{a}{c}+\frac{c}{a}\right).\]
Namely, for some positive constant $K_1$ and $b=b(t)>b_p$,
\[ b'\ge K_1b^3,\]
showing that $\eta$ is finite, and so $b\to\infty$ as $t\to\eta^-$. Additionally, we see that for some constant $K_2>0$, denoting $b_0:=\max\{b(t_0),b_p\}>0$, we have
\begin{equation}\label{eta_dist} \lim_{t\to \eta}d_g(p_0,M^t)= \int_{t_0}^\eta abc\, dt \ge\int_{b_0}^\infty \frac{2}{\frac{a}{c}+\frac{c}{a}}\,db \ge \int_{b_0}^\infty\frac{K_2}{b} db=\infty. \end{equation}
\end{proof}

\subsection{The Bolt}
Recall, for example from \cite{mr}, that ``attaching a bolt" refers to replacing a $4$-manifold with a cohomogeneity one action with only regular orbits over an open interval with one admitting
a similar action for the same group over a semi-closed interval with a $2$-dimensional singular orbit (the bolt) over the endpoint of the interval.
For the case at hand, the latter $4$-manifold can then be described as
\[
E(2)\times_{SO(2)}\mathbb{R}^2= (0,\infty)\times E(2)\ \amalg\ \{0\}\times \mathbb{R}^2 ,
\]
where the right $SO(2)$-action is $(g, (T,x))\to (Tg,g^{-1}x)$. So this is an instance of
the set-up described in the beginning of Section~\ref{sec:coho}.
\begin{prop}\label{prop:bolt}
The metric and K\"ahler form corresponding to solutions of \Ref{E2ODEa}-\Ref{E2ODEc} along the unstable curves of the equilibrium points $(q,0,q)$, $q>0$, defined on $(-\infty, \eta)$,
can be smoothly extended to $E(2)\times_{SO(2)}\mathbb{R}^2$, where the bolt fibers over $\xi=-\infty$.
\end{prop}
\begin{proof}
Consider $M=E(2)\times_{SO(2)}\mathbb{R}^2$.  This has a left action by $E(2)$ with regular orbit $E(2)$ and a singular orbit $E(2)/SO(2)$.
For any $E(2)$-invariant metric $g$ on $M$, with $r$ the distance along a geodesic perpendicular to the singular orbit,
\[ g = dr^2+g_r. \]
For a metric $g$ of the form \Ref{bianchiAmet1}, let $r=\int_{-\infty}^t a(s)b(s)c(s)\,ds$, then
\[ g = dr^2+a^2\sigma_1^2+b^2\sigma_2^2+c^2\sigma_3^2. \]
The ODE's \Ref{E2ODEa}-\Ref{E2ODEc} in this coordinate become
\begin{align}
\label{rODEa}\frac{da}{dr}&=\frac{a}{2}\left(-\frac{a}{bc}+\frac{c}{ab}\right),\\
\label{rODEb}\frac{db}{dr}&=\frac{1}{2}\left(\frac{a}{c}+\frac{c}{a}\right),\\
\label{rODEc}\frac{dc}{dr}&=\frac{c}{2}\left(\frac{a}{bc}-\frac{c}{ab}+2\frac{ab}{c}
+2\fr{a\ep(b)}{bc}\right).
\end{align}
From these it is seen that if $\ep(b)$ is extended as an even function of $b$, then $a,b,$ and $c$ can be extended at $r=0$ so that $a$ and $c$ are even and $b$ is odd.
Following the notations of Verdiani and Ziller \cite{vz}, the tangent space for $r\neq 0$ splits as
\[ T_p M = \mathbb{R}\partial_r\oplus \mathfrak{k}\oplus\mathfrak{m}, \]
where
\[\mathfrak{k}=\mathrm{span}\{X_2\}, \]
\[ \mathfrak{m}=\mathrm{span}\{X_1,X_3\}=:\ell_1, \]
and we set
\[ V=\mathrm{span}\{\partial_r,X_2\}=:\ell_{-1}'. \]
Since $\exp(\theta X_2)$ acts on both $V$ and $\mathfrak{m}$ as a rotation by $\theta$, we have weights $a_1=d_1=1$ (with the former not to be confused with its use in previous sections).
The smoothness conditions for $V$ is that $b$ can be extended to an odd function and $b'(0)=1$.
Since we know that $b$ can be extended to be odd, we just check from \Ref{rODEb} that
\[ \left.\frac{db}{dr}\right|_{r=0}=\frac{1}{2}\left(\frac{q}{q}+\frac{q}{q}\right)=1.\]
As $\ell'_{-1}$ and $\ell_1$ are orthogonal, the smoothness conditions in table C of
\cite{vz} are automatically satisfied, while those in
table B there, are
\begin{align}\label{smooth1}a^2+c^2&=\phi_1(r^2),\\
\label{smooth2}a^2-c^2&=r^2\phi_2(r^2),
\end{align}
for some smooth functions $\phi_1$ and $\phi_2$.
Now to see that \Ref{smooth1} is satisfied, note that
\[ a^2+c^2=2ac\frac{db}{dr}.\]
Since $a$, $c$, and $\frac{db}{dr}$ are even, it just remains to check \Ref{smooth2}.
As $a$, $c$ are even while $b$ is odd, $a/c$ is even as a function of $b$.
We have to second order
\[ a^2-c^2=c^2\left(\frac{a^2}{c^2}-1\right)\propto b^2c^2, \]
and since $b(0)=0$ and $\frac{db}{dr}|_{r=0}=1$, we get that $g$ extends to a smooth metric on $M$.

Finally we check that the K\"ahler form also extends smoothly to the singular orbit at $r=0$.
Following Verdiani and Ziller we analyze the eigenspaces for the action of $SO(2)$ on $T_pM$.
We find that $\partial_r+\frac{i}{r}X_2$ is an eigenvector with eigenvalue $e^{ia_1\theta}$, and
$X_1+iX_3$ is an eigenvector with eigenvalue $e^{id_1\theta}$, and likewise for their complex conjugates.
Dualizing gives eigenspaces of $T^*_pM$: $dr-ir \sigma_2$ has eigenvalue $e^{ia_1\theta}$, and $\sigma_1-i\sigma_3$ has eigenvalue  $e^{id_1\theta}$.
Thus the eigenspaces of $\Lambda^2T^*_pM$ are given by
\begin{align*}
E_1 &= \mathrm{span}\{r dr\wedge \sigma_2, \sigma_1\wedge\sigma_3\}\\
E_{e^{i(a_1-d_1)\theta}}&=\mathrm{span}\{dr\wedge\sigma_1+r\sigma_2\wedge\sigma_3+i(dr\wedge\sigma_3+r\sigma_1\wedge\sigma_2)\}\\
E_{e^{i(a_1+d_1)\theta}}&=\mathrm{span}\{dr\wedge\sigma_1-r\sigma_2\wedge\sigma_3+i(dr\wedge\sigma_3-r\sigma_1\wedge\sigma_2)\}\\
&\qquad
\end{align*}
According to \cite{vz}, the smoothness condition is just the equivariance condition $\omega(e^{a_1\theta}p)=\exp(\theta X_2)^*\omega$.
This requires that the coefficient of
\begin{align}
E_1 & \text{ is } \phi_1(r^2), \\
E_{e^{\pm i(a_1-d_1)\theta}}& \text{ is } r^{\frac{|a_1-d_1|}{a_1}}\phi_2(r^2),\\
E_{e^{\pm i(a_1+d_1)\theta}}& \text{ is } r^{\frac{|a_1+d_1|}{a_1}}\phi_3(r^2).
\end{align}
But we have
\begin{align*}\omega&=cdr\wedge\sigma_3+ab\sigma_1\wedge\sigma_3\\
&=\frac{c}{2}[(dr\wedge\sigma_3+r\sigma_1\wedge\sigma_2)+(dr\wedge\sigma_3-r\sigma_1\wedge\sigma_2)]\\
&\qquad + \frac{ab}{2r}[(dr\wedge\sigma_3+r\sigma_1\wedge\sigma_2)-(dr\wedge\sigma_3-r\sigma_1\wedge\sigma_2)]\\
&=\frac{cr+ab}{2r}(dr\wedge\sigma_3+r\sigma_1\wedge\sigma_2) \\
&\qquad + \frac{cr-ab}{2r}(dr\wedge\sigma_3-r\sigma_1\wedge\sigma_2),
\end{align*}
and therefore the smoothness conditions can now be written as
\begin{align*}
cr+ab &= r\phi_2(r^2), \\
cr-ab &= r^3\phi_3(r^2). \\
\end{align*}
The first of these is clear; for the second, expand to get $a/c= 1+\al b^2+\mathcal{O}(b^4)$
for some constant $\al$ and $b= r+\mathcal{O}(r^3)$, so
\[ cr\left(1-\frac{ab}{cr}\right)=cr\left(1-\frac{b+\al b^3
+\mathcal{O}(b^5)}{r}\right)=r^3\phi_3(r^2).\]
Therefore $\omega$ extends as a smooth form on all of $M$.
\end{proof}

\subsection{Completeness}
\begin{prop}\label{prop:complete}
For the metrics of Proposition~\ref{prop:bolt},
all finite length curves remain inside some compact set.
\end{prop}

The proof is exactly the same as in \cite{mr}.
This completes the proof of Theorem~\ref{complete}.

\appendix
\section{The K\"ahler condition}

We review here the K\"ahler condition in terms of the decomposition~\Ref{decomp}.
For a K\"ahler form $\om$ on a manifold $M^m$, $m>1$, the vanishing of $d\om$ on orthonormal
frame vector fields $\{e_{2j-1}, e_{2j}\}\subset\mathcal{H}_j$, $j=1,\ldots m$ amounts to
\[
{2m\choose 3}=2m(m-1)+2^3{m\choose 3}
\]
independent equations, with $2m(m-1)$ involving frame fields only from to two $J$-invariant
subbundles $\mathcal{H}_i$, $\mathcal{H}_j$, $i\ne j$ while the remaining
$8{m\choose 3}$ equations involve vector fields belonging to three
subbundles $\mathcal{H}_i$, $\mathcal{H}_j$, $\mathcal{H}_k$, $i\ne j\ne k\ne i$
(and ${m\choose 3}$ is taken to be zero if $m=2$).
For each fixed {\em ordered} pair $i$, $j$ there are two equations of the first of
these two types, the first having the form
\be\lb{2-subs}
\hat{e}_{2j}([e_{2i-1},e_{2i}])+\hat{e}_{2i-1}([e_{2i-1},e_{2j-1}])+\hat{e}_{2i}([e_{2i},e_{2j-1}])=0
\end{equation}
while the second obtained by exchanging $e_{2j-1}$ with $e_{2j}$ (and their hatted counterparts)
in the first one. A typical example of the second type has the form
\be\lb{3-subs}
\hat{e}_{2k}([e_{2i-1},e_{2j-1}])+\hat{e}_{2i}([e_{2j-1},e_{2k-1}])+\hat{e}_{2j}([e_{2k-1},e_{2i-1}])=0.
\end{equation}
All these facts are easily obtained from the coboundary formula for the exterior derivative
of a $2$-form. Note that the {\em sum} ${2m\choose 3}+2m(m-1)$, where the second term now corresponds to the number of  vector equations needed for integrability of the complex structure,
yields a quantity still smaller than the $2m{2m\choose 2}$ unknowns needed to specify all Lie
brackets of pairs of frame vector fields.

For the frame with Lie brackets given by the first paragraph of subsection~\ref{construct},
it is clear that equations of type \Ref{3-subs} will hold. The same is true for equations of type
\Ref{2-subs}, which require knowledge of $[\Gamma(\mathcal{H}_i),\Gamma(\mathcal{H}_j)]$ for
$i=j$, as opposed to integrability of the complex structure, which only required such bracket
relations with $i\ne j$. For example, we have
\begin{align*}
\hat{\yy}_i([\kk,\tT])&+\hat{\kk}([\kk,\xx_i])+\hat{\tT}([\tT,\xx_i])\\
&=\hat{\yy}_i(L(\kk+\tT))+\hat{\kk}(A_i\xx_i+B_i\yy_i)+\hat{\tT}(E_i\xx_i+F_i\yy_i)\\
&=0,\\
\hat{\tT}([\xx_i,\yy_i])&+\hat{\xx}_i([\xx_i,\kk])+\hat{\yy}_i([\yy_i,\kk])\\
&=\hat{\tT}(N_i(\kk+\tT))+\hat{\xx}_i(-A_i\xx_i-B_i\yy_i)+\hat{\yy}_i(-C_i\xx_i-D_i\yy_i)\\
&=N_i-A_i-D_i=0,
\end{align*}
etc. This provides another proof for Prop.~\ref{kah}, and explains the need to require
relations~\Ref{rels2}. One can also draw some limited conclusions regarding possible Lie bracket
values for K\"ahler structures. For example, if $m=2$, we see that with an indexing
$\mathcal{H}_0=\mathrm{span}(\kk,\tT)$, $\mathcal{H}_1=\mathrm{span}(\xx_{\1},\yy_{\1})$
and shear relations as in \Ref{brack2}-\Ref{brack3} and \Ref{rels1},
$[\Gamma(\mathcal{H}_0),\Gamma(\mathcal{H}_0)]$ cannot have a component in $\mathcal{H}_1$,
whereas $[\Gamma(\mathcal{H}_1),\Gamma(\mathcal{H}_1)]$ can have a component in $\mathcal{H}_1$,
without preventing the metric from being K\"ahler.

\section*{Acknowledgements}
We thank Ramiro Lafuente, whose message alerted us to an incorrect description in an earlier version of the manifold in Theorem~\ref{sol-Heis2}.  We acknowledge Ronan Conlon for pointing us at a late stage to
the reference \cite{r} and for some very useful exchanges, and Ali Maalaoui for remarks regarding holomorphic sectional curvature.

\end{document}